\documentclass[12pt, twoside]{article}

\usepackage{amssymb, amsmath, amsthm, graphicx,enumitem,color}
\usepackage[left=1in,top=1in,right=1in]{geometry}

\usepackage{subfigure}

      \newtheorem{theorem}{Theorem}[section]
      \newtheorem{lemma}[theorem]{Lemma}

      \newtheorem{observation}[theorem]{Observation}
      \newtheorem{corollary}[theorem]{Corollary}

      \newtheorem*{partitionVThmEnv}{Theorem \ref{th:partitionV}}
      \newtheorem*{IncRdThmEnv}{Theorem \ref{th:IncRd}}

\theoremstyle{definition}

\theoremstyle{remark}

	\newcommand{\RR}{{\mathbb R}}

\renewcommand{\dim}{\operatorname{dim}}

\newcommand{\RP}{\mathbb{R} \mathbf{P}}
\newcommand{\CC}{\mathbb C}

\newcommand{\pts}{P}

\def\eps{{\varepsilon}}
\newcommand{\qvec}[1]{\textbf{\textit{#1}}}

\def\cp{{C_\text{part}}}
\def\cells{{C_\text{cells}}}
\def\inter{{C_\text{inter}}}
\def\chol{{C_\text{H\"old}}}

\newcommand{\vrts}{\mathcal V}
\newcommand{\curves}{\Gamma}

\newcommand{\ignore}[1]{}

\def\subjclass#1{{\renewcommand{\thefootnote}{}%
\footnote{\emph{Mathematics Subject Classification (2010):} #1}}}
\def\keywords#1{\par\medskip
\noindent\textbf{Keywords.} #1}

\begin{document}

\title{A semi-algebraic version of Zarankiewicz's problem}
\author{Jacob Fox\thanks{Stanford University, Stanford, CA. Supported by a Packard Fellowship, by NSF CAREER award DMS 1352121, and by an Alfred P. Sloan Fellowship. Email: {\tt jacobfox@stanford.edu}.} \and J\'anos Pach\thanks{EPFL, Lausanne and Courant Institute, New York, NY. Supported
by Hungarian Science Foundation EuroGIGA Grant OTKA NN 102029, by Swiss National Science Foundation Grants 200020-144531 and 200021-137574. Email:
{\tt pach@cims.nyu.edu}.}\and Adam Sheffer\thanks{Corresponding author. California Institute of Technology, Pasadena, CA. Supported by Grant 338/09 from the Israel Science Fund and by the Israeli Centers of Research Excellence (I-CORE) program (Center No. 4/11). Email: {\tt adamsh@gmail.com}
.} \and Andrew Suk\thanks{Massachusetts Institute of Technology, Cambridge, MA. Supported by an NSF Postdoctoral
Fellowship and by Swiss National Science Foundation Grant 200021-137574. Email: {\tt asuk@math.mit.edu}.} \and Joshua Zahl\thanks{Massachusetts Institute of Technology, Cambridge, MA. Supported by an NSF Postdoctoral Fellowship. Email: {\tt jzahl@mit.edu}}
}

\date{}

\maketitle


\begin{abstract}
A bipartite graph $G$ is \emph{semi-algebraic in} $\mathbb{R}^d$ if its vertices are represented by point sets $P,Q \subset \mathbb{R}^d$ and its edges are defined as pairs of points $(p,q) \in P\times Q$ that satisfy a Boolean combination of a fixed number of polynomial equations and inequalities in $2d$ coordinates.  We show that for fixed $k$, the maximum number of edges in a $K_{k,k}$-free semi-algebraic bipartite graph $G = (P,Q,E)$ in $\mathbb{R}^2$ with $|P| = m$ and $|Q| = n$ is at most $O((mn)^{2/3} + m + n)$, and this bound is tight. In dimensions $d \geq 3$, we show that all such semi-algebraic graphs have at most $C\left((mn)^{  \frac{d}{d+1} + \eps} + m + n\right)$ edges, where here $\eps$ is an arbitrarily small constant and $C = C(d,k,t,\eps)$. This result is a far-reaching generalization of the classical Szemer\'edi-Trotter incidence theorem. The proof combines tools from several fields: VC-dimension and shatter functions, polynomial partitioning, and Hilbert polynomials.

We also present various applications of our theorem. For example, a general point-variety incidence bound in $\RR^d$, an improved bound for a $d$-dimensional variant of the Erd\H os unit distances problem, and more.


\end{abstract}

\section{Introduction}

The problem of Zarankiewicz \cite{Za} is a central problem in extremal graph theory. It asks for the maximum number of edges in a bipartite graph which has $m$ vertices in the its first class, $n$ vertices in the second class, and does not contain the complete bipartite graph $K_{k,k}$ with $k$ vertices in each part. In 1954, K\H ov\'ari, S\'os, and Tur\'an \cite{kovari} proved a general upper bound of the form $c_k(mn^{1-1/k}+n)$ edges, where $c_k$ only depends on $k$.
Well-known constructions of Reiman and Brown shows that this bound is tight for $k =  2,3$ (see \cite{pach}).  However, the Zarankiewicz problem for $k \geq 4$ remains one of the most challenging unsolved problems in extremal graph theory.  A recent result of Bohman and Keevash \cite{BK} on random graph processes gives the best known lower bound for $k \geq 5$ and $m=n$ of the form  $\Omega\left(n^{2- 2/(k+1)}(\log k)^{1/(k^2-1)}\right)$.
In this paper, we consider Zarankiewicz's problem for \emph{semi-algebraic}\footnote{A real semi-algebraic set in $\mathbb{R}^{d_1+d_2}$ is the locus of all points that satisfy a given finite Boolean combination of polynomial equations and inequalities in the $d_1+d_2$ coordinates.} bipartite graphs, that is, bipartite graphs where one vertex set is a collection of points in $\RR^{d_1}$, the second vertex set is a collection of points in $\RR^{d_2}$, and edges are defined as \emph{pairs} of points that satisfy a Boolean combination of polynomial equations and inequalities in $d_1+d_2$ coordinates. This framework captures many of the well-studied incidence problems in combinatorial geometry (see, e.g., \cite{PaSh}).

Let $G = (P,Q,E)$ be a semi-algebraic bipartite graph in $(\mathbb{R}^{d_1},\RR^{d_2})$ with $|P| =m$ and $|Q| = n$.  Then there are polynomials $f_1,f_2,\ldots,f_{t} \in \mathbb{R}[x_1,\ldots,x_{d_1+d_2}]$ and a Boolean function $\Phi(X_1,X_2,\ldots,X_{t})$ such that for $(p,q) \in P\times Q \subset \mathbb{R}^{d_1}\times \mathbb{R}^{d_2}$,

\begin{equation*}
(p,q) \in E \hspace{.5cm}\Leftrightarrow \hspace{.5cm}\Phi\big(f_1(p,q) \geq 0,\ldots,f_t(p,q) \geq 0\big) = 1.
\end{equation*}

We say that the edge set $E$ has \emph{description complexity} at most $t$ if $E$ can be described with at most $t$ polynomial equations and inequalities, and each of them has degree at most $t$.  If $G = (P,Q,E)$ is $K_{k,k}$-free, then by the  K\H ov\'ari-S\'os-Tur\'an theorem we know that $|E(G)|=~O(mn^{1-1/k}~+~n)$. However, our main result gives a much better bound if $G$ is semi-algebraic of bounded description complexity.  In particular, we show that Zarankiewicz's problem for semi-algebraic bipartite graphs primarily depends on the dimension.

\begin{theorem}
\label{2dim}
Let $G = (P,Q,E)$ be a semi-algebraic bipartite graph in $(\mathbb{R}^{d_1},\mathbb{R}^{d_2})$ such that $E$ has description complexity at most $t$, $|P| = m$, and $|Q| = n$.  If $G$ is $K_{k,k}$-free, then

\begin{align}
&|E(G)| \leq c_1\left((mn)^{2/3}   + m + n\right)\quad\textrm{for}\ d_1=d_2 = 2,\label{d1d22Bd}\\
&|E(G)| \leq c_2\left((mn)^{d/(d+1)+\eps}+ m + n\right)\quad\textrm{for}\ d_1=d_2 = d,\label{d1d2dBd}
\end{align}
\noindent and more generally,
\begin{align}
&|E(G)| \leq c_3\left(m^{\frac{d_2(d_1-1)}{d_1d_2-1}+\eps}n^{\frac{d_1(d_2-1)}{d_1d_2-1}} + m + n\right)\quad\textrm{for all}\ d_1,d_2.\label{d1d2GenBd}
\end{align}

\noindent Here, $\eps$ is an arbitrarily small constant and $c_1 = c_1(t,k), c_2 = c_2(d,t,k,\eps), c_3 = c_3(d_1,d_2,t,k,\eps)$.
\end{theorem}

\ignore{ 
{\bf Lower bounds.} When $d_1=d_2=2$ or $d_1=d_2=3$ and $m\approx n$, Theorem \ref{2dim} is tight up to the $\eps$ loss in the exponent and to the constant $c$. To see this we set $d=d_1=d_2$, and consider a large positive integer $N$ and the following two sets of integers.
\begin{align*}
P &= \left\{ (p_1,\ldots,p_{d}) \mid 1 \le p_1,\ldots,p_{d-1} \le N \text{ and } 1 \le p_{d} \le N^2 \right\}, \\[2mm]
Q &= \left\{ (q_1,\ldots,q_{d}) \mid 1 \le q_1,\ldots,q_{d-1} \le N \text{ and } -N^2 \le q_{d} \le dN^2 \right\}.
\end{align*}
Notice that we have $m=|P| = N^{d+1}$ and $n=|Q| = \Theta_d (N^{d+1})$.
Let $G$ be the semi-algebraic graph defined by the functions
\begin{align*}
f_1(p,q)& =(p_{d}-q_{d})-\left((p_1-q_1)^2+\ldots+ (p_{d-1}-q_{d-1})^2\right), \\[2mm]
f_2(p,q) &= -f_1(p,q).
\end{align*}
For every $p\in P$, there at least $N^{d-1}$ points $q\in Q$ with $f_1(p,q)=f_2(p,q)=0$. Thus, the number of edges of the graph $G$ is $\Omega(N^{2d}) = \Omega\left( (mn)^{d/(d+1)}\right)$. We can verify that the graph $G$ has no $K_{d+1,d+1}$; when $d=2$, this is just the statement that two points define a unit parabola (i.e.,~a translate of the graph of the function $y=x^2$), and when $d=3$, three points define a unit paraboloid. Thus, the graph $G$ satisfies the hypothesis of Theorem \ref{2dim}. When $d_1\neq d_2$, or when $d_1=d_2=d$ and $d>3$, there are no known lower bounds to match the upper bounds from Theorem \ref{2dim}. }

To prove the theorem, we combine ideas from the study of VC-dimension with ideas from incidence theory. In the latter, we rely on the concept of polynomial partitioning (as introduced by Guth and Katz \cite{GK15}) and combine it with a technique that relies on Hilbert polynomials. Recently, similar polynomial partitioning techniques were also studied by Matou\v{s}ek and Safernov\'a \cite{MS14} and Basu and Sombra \cite{BasuSombra14}. However, each of the three papers presents different proofs and very different results.

The planar case of Theorem \ref{2dim} (i.e., \eqref{d1d22Bd})
is a generalization of the famous Szemer\'edi-Trotter point-line theorem \cite{ST83}.
Indeed, in the case of $d_1=d_2=2$, by taking $P$ to be the point set, $Q$ to be the dual of
the lines, and the relationship to be the incidence relationship, we get
that $G$ is $K_{2,2}$-free as two distinct lines intersect in at most one point. As we will see below, there
are many further applications of Theorem \ref{2dim}.

\medskip

\textbf{Previous work and lower bounds.}  Several authors have studied this extremal problem in a more restricted setting: on bounding the number of incidences between an $m$-element point set $P$ and a set of $n$ hyperplanes $H$ in $\mathbb{R}^d$ where no $k$ points of $P$ lies on $k$ hyperplanes of $H$.  Since each hyperplane $h\subset \mathbb{R}^d$ dualizes\footnote{Given a hyperplane $h = \{(x_1,\ldots,x_d): a_1x_1 + \cdots + a_dx_d = 1\}$ in $\mathbb{R}^d$, the dual of $h$ is the point $h^{\ast} = (a_1,\ldots,a_d)$.} to a point in $\mathbb{R}^d$, this problem is equivalent to determining the maximum number of edges in a $K_{k,k}$-free semi-algebraic bipartite graph $G = (P,Q,E)$ in $(\mathbb{R}^d,\mathbb{R}^d)$, where $(p,q) \in E$ if and only if $\langle p,q\rangle =1$.  In this special case, the works of Chazelle \cite{chaz}, Brass and Knauer \cite{brass}, and Apfelbaum and Sharir \cite{sharir} implies that $|E(G)| \leq c'((mn)^{  \frac{d}{d+1}} +  m + n)$, where $c'$ depends only of $k$ and $d$.

On the other hand,  Brass and Knauer \cite{brass} gave a construction of an $m$-element point set $P$ and a set of $n$ hyperplanes $H$ in $\mathbb{R}^3$, with no $k$ points from $P$ lying on $k$ hyperplanes of $H$, with at least $\Omega((mn)^{7/10})$ incidences. For any $d\ge 4$ and $\varepsilon>0$, Sheffer \cite{Sheffer15} presented a construction of an $m$-element point set $P$ and a set of $n=\Theta(m^{(3-3\varepsilon)/(d+1)})$ hyperplanes $H$ in $\mathbb{R}^d$, with no two points from $P$ lying on $(d-1)/\eps$ hyperplanes of $H$, with $\Omega((mn)^{1-\frac{2}{d+4}-\eps})$ incidences.
These are the best known lower bounds for Theorem \ref{2dim} that we are aware of. Notice that gap between these bounds and the upper bound of \eqref{d1d2dBd} becomes rather small for large values of $d$.
\medskip

\textbf{Applications.} After proving Theorem \ref{2dim}, we provide a variety of applications.
First, we show how a minor change in our proof leads to the following general incidences bound.

\begin{theorem} \label{th:IncRd}
Let $\pts$ be a set of $m$ points and let $\vrts$ be a set of $n$ constant-degree algebraic varieties, both in $\RR^d$, such that the incidence graph of $\pts\times\vrts$ does not contain a copy of $K_{s,t}$ (here we think of $s,t,$ and $d$ as being fixed constants, and $m$ and $n$ are large). Then for every $\eps>0$, we have
$$I(\pts,\vrts) = O\left(m^{\frac{(d-1)s}{ds-1}+\eps}n^{\frac{d(s-1)}{ds-1}}+m+n\right).$$
\end{theorem}
Theorem \ref{th:IncRd} subsumes many known incidences results (up to the extra $\eps$ in the exponent), and extends them to $\RR^d$ (see Section \ref{applicationSection}). When $s=2$, the theorem is tight up to subpolynomial factors (see \cite{Sheffer15}).
We also derive an improved bound for a $d$-dimensional variant of the Erd\H os unit distances problem, a bound for incidences between points and tubes, and more.
\medskip

\textbf{Organization.}  In Section \ref{vcdim}, we give an upper bound on the maximum number of edges in a $K_{k,k}$-free bipartite graph with bounded VC-dimension.  In Section \ref{warmup}, we establish the bound \eqref{d1d22Bd} from Theorem \ref{2dim}. Then in Section \ref{polyMethodSection}, we prove the bounds \eqref{d1d2dBd} and \eqref{d1d2GenBd} from Theorem \ref{2dim}. The parts of this proof that concern Hilbert polynomials are deferred to Section \ref{sec:hilbert}.
In Section \ref{applicationSection}, we discuss applications of Theorem \ref{2dim}.
Finally, Section \ref{sec:disc} consists of a brief discussion concerning the tightness of our results.
\medskip

\textbf{Acknowledgments.}
The authors would like to thank Saugata Basu for comments and corrections to an earlier version of this manuscript and G\'abor Tardos for many valuable discussions. Work on this paper was performed while the authors were visiting the Institute for Pure and Applied Mathematics (IPAM), which is supported by the National Science Foundation.

\section{VC-dimension and shatter functions}\label{vcdim}

Given a bipartite graph $G = (P,Q,E)$ where $E\subset P\times Q$, for any vertex $q \in Q$, let $N_G(q)$ denote the neighborhood of $q$ in $G$, that is, the set of vertices in $P$ that are connected to $q$.  Then let $\mathcal{F} = \{N_G(q) \subset P: q \in Q\}$ be a set system with ground set $P$.  The dual of $(P,\mathcal{F})$ is the set system obtained by interchanging the roles of $P$ and $\mathcal{F}$, that is, it is the set system $(\mathcal{F},\mathcal{F}^{\ast})$, where $\mathcal{F}$ is the ground set and $\mathcal{F}^{\ast} = \{\{A\in \mathcal{F}: p\in A\}:p \in P\}$.  Obviously, $(\mathcal{F}^{\ast})^{\ast} = \mathcal{F}$.

The {\em Vapnik-Chervonenkis dimension} (in short, VC-dimension) of $(P,\mathcal{F})$ is the {\em largest} integer $d_0$ for which there exists a $d_0$-element set $S\subset P$ such that for every subset $B\subset S$, one can find a member $A\in \mathcal{F}$ with $A\cap S=B$.  The \emph{primal shatter function} of $(P,\mathcal{F})$ is defined as

\begin{equation*}
\pi_{\mathcal{F}}(z) = \max\limits_{P'\subset P, |P'| = z}|\{A\cap P':A \in \mathcal{F}\}|.
\end{equation*}

\noindent In other words, $\pi_{\mathcal{F}}(z)$ is a function whose value at $z$ is the maximum possible number of distinct intersections of the sets of $\mathcal{F}$ with a $z$-element subset of $P$.  The primal shatter function of $\mathcal{F}^{\ast}$ is often called the \emph{dual shatter function} of $\mathcal{F}$.

The VC-dimension of $\mathcal{F}$ is closely related to its shatter functions.  A result of Sauer and Shelah states that if $\mathcal{F}$ is a set system with VC-dimensions $d_0$, then

\begin{equation} \label{eq:SauerShelah}
\pi_{\mathcal{F}}(z) \leq   \sum_{i=0}^{d_0}{z\choose i}.
\end{equation}

\noindent On the other hand, suppose that the primal shatter function of $\mathcal{F}$ satisfies $\pi_{\mathcal{F}}(z) \leq cz^d$ for all $z$.  Then, if the VC-dimension of $\mathcal{F}$ is $d_0$, we have $2^{d_0} \leq c(d_0)^d$, which implies $d_0 \leq 4d\log (cd)$.

Most of this section is dedicated to proving the following result.

\begin{theorem}
\label{weak}

Let $G = (P,Q,E)$ be a bipartite graph with $|P| = m$ and $|Q| = n$ such that the set system $\mathcal{F}_1 = \{N(q): q\in Q\}$ satisfies $\pi_{\mathcal{F}_1}(z) \leq cz^d$ for all $z$.  Then, if $G$ is $K_{k,k}$-free, we have

\begin{equation*}
|E(G)| \leq c_1(mn^{1 - 1/d} + n),
\end{equation*}

\noindent where $c_1 = c_1(c,d,k)$.
\end{theorem}

Let $f_1, \ldots, f_{\ell}$ be $d$-variate real polynomials with respective zero-sets $V_1, \ldots, V_{\ell}$; that is, $V_i = \{x \in \mathbb{R}^d: f_i(x) = 0\}$.  A vector $\sigma\in \{-1, 0, +1\}^{\ell}$ is a {\it sign pattern} of $f_1, \ldots, f_{\ell}$ if there exists an $x \in \mathbb{R}^d$ such that the sign of $f_j(x)$ is $\sigma_j$ for all $j = 1,\ldots, {\ell}$. The Milnor-Thom theorem (see \cite{basu, milnor, thom}) bounds the number of cells in the arrangement of the zero-sets $V_1, \ldots, V_{\ell}$ and, consequently, the number of possible sign patterns.

\begin{theorem}[Milnor-Thom]\label{milnor}
Let $f_1,\ldots,f_{\ell}$ be $d$-variate real polynomials of degree at most $t$.  The number of cells in the arrangement of their zero-sets $V_1,\ldots,V_{\ell}\subset \mathbb{R}^d$ and, consequently, the number of sign patterns of $f_1,\ldots,f_{\ell}$ is at most

\begin{equation*}
\left(\frac{50t\ell}{d}\right)^d
\end{equation*}

\noindent for $\ell \geq d \geq 2$.
\end{theorem}

We have the following consequence of Theorems \ref{weak} and \ref{milnor}.

\begin{corollary}\label{app}
Let $G = (P,Q,E)$ be a bipartite semi-algebraic graph in $(\mathbb{R}^{d_1},\RR^{d_2})$ with $|P| = m$ and $|Q| = n$ such that $E$ has complexity at most $t$.  If $G$ is $K_{k,k}$-free, then

\begin{equation*}
|E(G)| \leq c'(mn^{1- 1/d_2} + n),
\end{equation*}

\noindent where $c' = c'(d_1,d_2,t,k)$.

\end{corollary}

\begin{proof}

Let $\mathcal{F}_1 = \{N(q): q\in Q\}$ and $\mathcal{F}_2 = \{N(p): p \in P\}$.  By Theorem \ref{weak}, it suffices to show that $\pi_{\mathcal{F}_1}(z) \leq cz^{d_2}$ for all $z$ and a constant $c = c(d_1,d_2,t,k)$.

Since $E$ is semi-algebraic, there are polynomials $f_1,\ldots,f_t$ and a Boolean formula $\Phi$ such that for $(p,q) \in P\times Q$,

\begin{equation*}
(p,q) \in E \hspace{.5cm}\Leftrightarrow \hspace{.5cm}\Phi(f_1(p,q)\geq 0,\ldots,f_t(p,q)\geq 0) = 1.
\end{equation*}

\noindent Notice the \emph{dual} of $\mathcal{F}_2$ is isomorphic to the set system $\mathcal{F}_1$.  Since any set of $z$ points $p_1,\ldots,p_z \in P$ corresponds to $z$ semi-algebraic sets $Z_1,\ldots,Z_z\subset \RR^{d_2}$ such that $Z_i = \{x \in \mathbb{R}^{d_2}: \Phi(f_1(p_i,x)\geq 0,\ldots,f_t(p_i,x) \geq 0) = 1\}$ and $N_G(p_i) = Q \cap Z_i$, by the Milnor-Thom theorem we have

\begin{equation*}
\pi_{\mathcal{F}_1}(z) = \pi_{\mathcal{F}^{\ast}_2}(z) \leq \left(\frac{50t^2z}{d_2}\right)^{d_2}.
\end{equation*}

\noindent This completes the proof of Corollary \ref{app}.

\end{proof}

The rest of this section is devoted to proving Theorem \ref{weak}, which requires the following lemmas.  Let $(P,\mathcal{F})$ be a set system on a ground set $P$.  The \emph{distance} between two sets $A_1,A_2 \in \mathcal{F}$ is $|A_1\bigtriangleup A_2|$, where $A_1 \bigtriangleup A_2 = (A_1\cup A_2)\setminus (A_1\cap A_2)$ is the symmetric difference of $A_1$ and $A_2$.  The \emph{unit distance graph} $UD(\mathcal{F})$ is the graph with vertex set $\mathcal{F}$, and its edges are pairs of sets $(A_1,A_2)$ that have distance one.  We will use the following result of Haussler.

\begin{lemma}[\cite{haus}]\label{unit}
If $\mathcal{F}$ is a set system of VC-dimension $d_0$ on a ground set $P$, then the unit distance graph $UD(\mathcal{F})$ has at most $d_0|\mathcal{F}|$ edges.

\end{lemma}

We say that the set system $\mathcal{F}$ is $(k,\delta)$-\emph{separated} if among any $k$ sets $A_1,\ldots,A_k \in \mathcal{F}$ we have

\begin{equation*}
|(A_1\cup \cdots \cup A_k)\setminus(A_1\cap \cdots \cap A_k)| \geq \delta.
\end{equation*}

\noindent The key tool used to prove Theorem \ref{weak} is the following \emph{packing lemma}, which was proved by Chazelle for set systems that are $(2,\delta)$-separated.  The proof of Lemma \ref{packing} can be regarded as a modification of Chazelle's argument (see \cite{mat2}), but we give a self-contained presentation.  We note that a weaker result, namely $|\mathcal{F}|\leq O\left((m/\delta)^d\log^d(m/\delta)\right)$, can be obtained with a simpler proof using epsilon-nets (see \cite{mat2} or \cite{szegedy}).

\begin{lemma}[Packing Lemma]\label{packing}

Let $\mathcal{F}$ be a set system on a ground set $P$ such that $|P| = m$ and $\pi_{\mathcal{F}}(z) \leq cz^d$ for all $z$.  If $\mathcal{F}$ is $(k,\delta)$-separated, then $|\mathcal{F}| \leq c'(m/\delta)^d$ where $c' = c'(c,d,k)$.

\end{lemma}

\begin{proof}
We assume, for contradiction, that $|\mathcal{F}| > c'(m/\delta)^d$ (where the constant $c'$ depends on $c,d,k$ and is set below).

Since the primal shatter function of $\mathcal{F}$ satisfies $\pi_{\mathcal{F}}(z) \leq cz^d$ for all $z$, we know that the VC-dimension of $\mathcal{F}$ is at most $4d\log (cd)$.  Set $d_0 = 4d\log (cd)$.   If  $\delta \leq 4k(k-1)d_0$, then the statement is trivial for sufficiently large $c'$ (by the assumption $|\mathcal{F}|\leq cm^d$).  Hence, we can assume  $\delta > 4k(k-1)d_0$.

Let $S\subset P$ be a random $s$-element subset, where $s  = \lceil 4k(k-1)d_0m/\delta \rceil$.  Set $\mathcal{T} = \{A\cap S: A \in \mathcal{F}\},$ and for each $B \in \mathcal{T}$ we define its \emph{weight} $w(B)$ as the number of sets $A \in \mathcal{F}$ with $A\cap S = B$.    Notice that

\begin{equation*}
\sum\limits_{B \in \mathcal{T}}w(B) = |\mathcal{F}|.
\end{equation*}

We let $E$ be the edge set of the unit distance graph $UD(\mathcal{T})$, and define the weight of an edge $e = (B_1,B_2)$ in $E$ as $\min(w(B_1),w(B_2))$.  Finally we set

\begin{equation*}
W = \sum\limits_{e \in E}w(e).
\end{equation*}

\noindent We will estimate the expectation of $W$ in two ways.

By Lemma \ref{unit}, we know that the unit distance graph $UD(\mathcal{T})$ has a vertex $B \in \mathcal{T}$ of degree at most $2d_0$. Since the weight of all edges emanating out of $B$ is at most $w(B)$, by removing vertex $B\in \mathcal{T}$, the total edge weight drops by at most $2d_0w(B)$.  By repeating this argument until there are no vertices left, we have

\begin{equation*}
W \leq 2d_0\sum\limits_{B \in \mathcal{T}} w(B) = 2d_0|\mathcal{F}|.
\end{equation*}

Now we bound $\mathbb{E}[W]$ from below.  Suppose we first choose a random $(s-1)$-element subset $S'\subset P$, and then choose a single element $p \in P\setminus S'$.  Then the set $S = S'\cup \{p\}$ is a random $s$-element set.  Let $E_1\subset E$ be the edges in the unit distance graph $UD(\mathcal{T})$ that differ by element $p$, and let

\begin{equation*}
W_1  = \sum\limits_{e \in E_1}w(e).
\end{equation*}

\noindent By symmetry, we have $\mathbb{E}[W] = s\cdot \mathbb{E}[W_1]$.  Hence, we shall bound $\mathbb{E}[W_1]$ from below.  To do so, we will estimate $\mathbb{E}[W_1|S']$ from below, which is the expected value of $W_1$ when $S'\subset P$ is a fixed $(s-1)$-element subset and we choose $p$ at random from $P\setminus S'$.

Divide $\mathcal{F}$ into equivalence classes $\mathcal{F}_1,\mathcal{F}_2,\ldots,\mathcal{F}_r$, where two sets $A_1,A_2 \in \mathcal{F}$ are in the same class if and only if $A_1\cap S' = A_2\cap S'$. By the assumption $\pi_{\mathcal{F}}(z) \leq cz^d$ for all $z$, we have

\begin{equation*}
r \leq \pi_{\mathcal{F}}(s-1) \leq c_0(m/\delta)^d,
\end{equation*}

\noindent where $c_0 = c_0(c,k,d)$.  Let $\mathcal{F}_i$ be one of the equivalence classes such that $|\mathcal{F}_i| = b$.  If an element $p \in P\setminus S'$ is chosen such that $b_1$ sets from $\mathcal{F}_i$ contain $p$ and $b_2 = b - b_1$ sets (from $\mathcal{F}_i$) do not contain $p$, then $\mathcal{F}_i$ gives rise to an edge in $E_1$ of weight $\min(b_1,b_2)$.  Since $\min(b_1,b_2) \geq b_1b_2/b$, we will estimate $\mathbb{E}[b_1b_2]$ from below when picking $p$ at random.  Notice that $b_1b_2$ is the number of ordered pairs of sets in $\mathcal{F}_i$ that differ in point $p$.  Hence,

\begin{equation}\label{b1b2}
\mathbb{E}[b_1b_2] \geq \sum\limits_{(A_1,A_2) \in \mathcal{F}_i\times\mathcal{F}_i} \mathbb{P}[p \in A_1\bigtriangleup A_2] = \sum\limits_{(A_1,A_2) \in \mathcal{F}_i\times\mathcal{F}_i}  \frac{ |A_1\bigtriangleup A_2|}{m - s + 1}.
\end{equation}

\noindent Now, given any $k$ sets $A_1,\ldots,A_k \in \mathcal{F}_i$, we have

\begin{equation*}
\bigcup\limits_{2\leq  j \leq k} A_1\bigtriangleup A_j =   (A_1\cup \cdots \cup A_k)\setminus(A_1\cap \cdots \cap A_k).
\end{equation*}

\noindent Since $\mathcal{F}_i$ is $(k,\delta)$-separated, we have

\begin{equation*}
\sum\limits_{2 \leq j  \leq k} |A_1 \bigtriangleup A_j| \geq |(A_1\cup \cdots \cup A_k)\setminus(A_1\cap \cdots \cap A_k)| \geq \delta.
\end{equation*}

\noindent Therefore, every $k$ sets in $\mathcal{F}_i$ contain a pair of sets $(A_1,A_j)$ such that $|A_1 \bigtriangleup A_j| \geq \delta/(k-1)$.  We define the auxiliary graph $G_i = (\mathcal{F}_i,E_i)$ whose vertices are the members in $\mathcal{F}_i$, and two sets $A_{1},A_{2} \in \mathcal{F}_i$ are adjacent if and only if $|A_{1}\bigtriangleup A_{2}| \geq \delta/(k-1)$.  Since $G_i$ does not contain an independent set of size $k$,  by Tur\'an's theorem (see, e.g.,  \cite{pach}), we have $|E_i| \geq \frac{b(b - k)}{2k}$.  Therefore,

\begin{equation}\label{turan}
\sum\limits_{(A_1,A_2) \in \mathcal{F}_i\times\mathcal{F}_i} |A_1 \bigtriangleup A_2|  \geq 2\frac{b(b - k)}{2k}\frac{\delta}{k-1}  =\frac{\delta}{k(k-1)}b(b - k).
\end{equation}

\noindent By combining equations (\ref{b1b2}) and (\ref{turan}), we have

\begin{equation*}
\mathbb{E}[b_1b_2] \geq  \frac{\delta}{k(k-1)m}b(b-k).
\end{equation*}

\noindent Since $\min(b_1,b_2)\geq b_1b_2/b$, the expected contribution of $\mathcal{F}_i$ to $W_1$ is at least $\frac{\delta}{k(k-1)m}(b-k)$.  Summing over all classes, we have

\begin{eqnarray*}
\mathbb{E}[W_1] & \geq & \frac{\delta}{k(k-1)m}\sum\limits_{i = 1}^r(|\mathcal{F}_i| - k) \\\\
 & = & \frac{\delta}{k(k-1)m}(|\mathcal{F}| - kr)\\\\
 & \geq & \frac{\delta}{k(k-1)m}(|\mathcal{F}| - kc_0(m/\delta)^d).
\end{eqnarray*}

\noindent Recall that $|\mathcal{F}| > c'(m/\delta)^d$. By taking $c'$ to be sufficiently large with respect to $k$ and $c_0$, and since $2d_0|\mathcal{F}| \geq \mathbb{E}[W] = s\cdot \mathbb{E}[W_1]$, we have

\begin{equation*}
2d_0|\mathcal{F}| \geq \frac{s\delta}{k(k-1)m}(|\mathcal{F}| - kc_0(m/\delta)^d)  \ge 4d_0|\mathcal{F}| - k4d_0c_0(m/\delta)^d,
\end{equation*}

\noindent which implies $|\mathcal{F}| \leq c'(m/\delta)^d$, where $c' = (c,d,k)$.
\end{proof}

We are now ready to prove Theorem \ref{weak}.

\medskip

\begin{proof}[Proof of Theorem \ref{weak}.]  Let $\mathcal{F}_1 = \{N(q): q\in Q\}$ and $\mathcal{F}_2 = \{N(p): p \in P\}$.  Notice the \emph{dual} of $\mathcal{F}_2$ is isomorphic to the set system $\mathcal{F}_1$.  Given a set of $k$ points $\{q_1,\ldots,q_k\} \subset Q$, we say that a set $B \in \mathcal{F}_2$ \emph{crosses} $\{q_1,\ldots,q_k\}$ if $ \{q_1,\ldots,q_k\} \cap B \neq \emptyset$ and $\{q_1,\ldots,q_k\} \not\subset B$.  We make the following observation.

\begin{observation}

There exists $k$ points $q_1,\ldots,q_k \in Q$ such that at most $2c'm/n^{1/d}$ sets from $\mathcal{F}_2$ cross $\{q_1,\ldots,q_k\}$, where $c'$ is defined in Lemma \ref{packing}.

\end{observation}

\begin{proof}
For the sake of contradiction, suppose that every set of $k$ points has at least $2c'm/n^{1/d}$ sets from $\mathcal{F}_2$ crossing it.  Then the dual set system $\mathcal{F}^{\ast}_2$ is $(k,\delta)$-separated, where $\delta = 2c'm/n^{1/d}$, and has the property that $\pi_{\mathcal{F}^{\ast}_2}(z) = \pi_{\mathcal{F}_1}(z) \leq cz^d$ for all $z$.  By Lemma \ref{packing}, we have

\begin{equation*}
n = |\mathcal{F}_2^{\ast}| \leq c'\left(\frac{m}{\delta}\right)^d.
\end{equation*}

\noindent Hence, $\delta \leq (c')^{1/d}m/n^{1/d}$, which is a contradiction.
\end{proof}

Let $q_1,\ldots,q_k$ be the set of $k$ points such that at most $2c'm/n^{1/d}$ sets in $\mathcal{F}_2$ cross it.  Since $G$ is $K_{k,k}$-free, there are at most $(k-1)$ points $p_1,\ldots,p_{k-1} \in P$ with the property that the neighborhood $N_G(p_i)$ contains $\{q_1,\ldots,q_k\}$, for $1\leq i \leq k-1$.  Therefore, the neighborhood of $q_1$ contains at most $2c'm/n^{1/d} + (k-1)$ points.  We remove $q_1$ and repeat this argument until there are fewer than $k$ vertices remaining in $Q$ and see that

\begin{equation*}
|E(G)| \leq (k-1)m + \sum\limits_{i = k}^n \left(2c'\frac{m}{i^{1/d}} + (k-1)\right) \leq c_1(mn^{1 - 1/d} + n),
\end{equation*}

\noindent for sufficiently large $c_1 = c_1(c,d,k)$.
\end{proof}

\section{The case where $d_1=d_2=2$}\label{warmup}
In this section, we shall prove Theorem \ref{2dim} in the case $d_1=d_2=2$, i.e., we shall establish part \eqref{d1d22Bd} of Theorem \ref{2dim}. Our argument will use the method of ``cuttings,'' which we shall now recall. Let $\Sigma = \{V_1,\ldots,V_n\}$ be a collection of curves of degree at most $t$ in $\mathbb{R}^2$, that is, $V_i = \{x \in \mathbb{R}^2: f_i(x) = 0\}$ for some bivariate polynomial $f_i$ of degree at most $t$.  We will assume that $t$ is fixed, and $n$ is some number tending to infinity.  A \emph{cell} in the arrangement $\mathcal{A}(\Sigma) = \bigcup_i V_i$ is a relatively open connected set defined as follows.  Let $\approx$ be an equivalence relation on $\mathbb{R}^2$, where $x \approx y$ if $\{i: x \in V_i\} = \{i:y \in V_i\}$.  Then the cells of the arrangement $\Sigma$ are the connected components of the equivalence classes. The classic Milnor-Thom Theorem says that the arrangement $\mathcal{A}(\Sigma)$ subdivides $\mathbb{R}^2$ into at most $O(n^2)$ cells (semi-algebraic sets),
but these cells can have very large description complexity.  A result of Chazelle et al.~\cite{chazelle} shows that these cells can be further subdivided into $O(n^2)$ smaller cells that have constant descriptive complexity. By combining this technique with the standard theory of random sampling \cite{agarwal,noga,shor}, one can obtain the following lemma which will be used in the next section.  We say that the surface $V_i = \{x \in \mathbb{R}^2: f_i(x) = 0\}$ \emph{crosses} the cell $\Omega\subset \mathbb{R}^2$ if $V_i\cap \Omega \neq \emptyset$ and $V_i$ does not fully contain $\Omega$.

\begin{lemma}{\bf (Cutting lemma, \cite{chazelle})}
\label{cut2}
For fixed $t > 0$, let $\Sigma$ be a family of $n$ algebraic surfaces in $\mathbb{R}^2$ of degree at most $t$.  Then for any $r > 0$, there exists a decomposition of $\mathbb{R}^2$ into at most $O(r^2)$ relatively open connected sets (cells) such that each cell is crossed by at most $n/r$ curves from $\Sigma$.
\end{lemma}

We are now ready to prove the following theorem, which will establish \eqref{d1d22Bd}.

\begin{theorem}
Let $G = (P,Q,E)$ be a semi-algebraic bipartite graph in $\mathbb{R}^2$ such that $E$ has description complexity at most $t$, $|P| = m$, and $|Q| = n$.  If $G$ is $K_{k,k}$-free, then

\begin{equation*}
|E(G)| \leq  c\left(m^{ \frac{2}{3}}n^{ \frac{2}{3}} + m + n\right),
\end{equation*}

\noindent where $c = c(k,t)$.

\end{theorem}

\begin{proof}
If $n > m^2$, then by Corollary \ref{app} we have $|E(G)| \leq (c/2)n$ for sufficiently large $c = c(k,t)$ and we are done.  Hence, we can assume $n \leq m^2$.  Since $E$ is semi-algebraic of description complexity at most $t$, there are polynomials $f_1,\ldots,f_t$ and a Boolean formula $\Phi$ such that for $(p,q) \in P\times Q$,

\begin{equation*}
(p,q) \in E \hspace{.5cm}\Leftrightarrow \hspace{.5cm}\Phi(f_1(p,q)\geq 0,\ldots,f_t(p,q)\geq 0) = 1.
\end{equation*}

\noindent  For each point $q\in Q$, let $V_{i,q} = \{x \in \mathbb{R}^2: f_i(x,q) = 0\}$, $1\leq i \leq t$.  Set $\Sigma = \{V_{i,q}: 1\leq i \leq t, q \in Q\}$.  Note that $|\Sigma| = tn$.

For $r = m^{2/3}/n^{1/3}$, we apply Lemma \ref{cut2}, the cutting lemma, to $\Sigma$, which partitions $\mathbb{R}^2$ into at most $c_2r^2$ cells $\Omega_i$, where $c_2 = c_2(t)$, such that each cell is crossed by at most $|\Sigma|/r$ surfaces from $\Sigma$.  By the Pigeonhole Principle, there is a cell $\Omega\subset \mathbb{R}^2$ that contains at least

\begin{equation*}
\frac{m}{c_2r^2} = \frac{n^{\frac{2}{3}}}{c_2m^{\frac{1}{3}}}
\end{equation*}

\noindent points from $P$.  Let $P'\subset P$ be a set of exactly $\lceil n^{\frac{2}{3}}/(c_2m^{\frac{1}{3}})\rceil$ points in $P\cap \Omega$. If $|P'| < k$, we have

\begin{equation*}
\frac{n^{\frac{2}{3}}}{c_2m^{\frac{1}{3}}}\leq |P'| < k,
\end{equation*}

\noindent which implies $m > n^2/(c_2^3 k^3)$.  By the dual of Corollary \ref{app}, we have $|E(G)| \leq (c/2)m$ for sufficiently large $c = c(k,t)$ and we are done.  Hence, we can assume $|P'| \geq k$.  Let $Q'\subset Q$ be the set of points in $Q$ that gives rise to a surface in $\Sigma$ that crosses $\Omega$.  By the cutting lemma,

\begin{equation*}
|Q'| \leq \frac{tn}{r} = t\frac{n^{\frac{4}{3}}}{m^{\frac{2}{3}}} \leq t(c_2)^2|P'|^2.
\end{equation*}

\noindent By Corollary \ref{app}, we have

\begin{equation*}
|E(P',Q')| \leq c'(|P'||Q'|^{1/2} + |Q'|) \leq c_3|P'|^2,
\end{equation*}

\noindent where $c'$ is defined in Corollary \ref{app}, and $c_3 = c_3(k,t)$.  Hence, there is a point $p \in P'$ such that $p$ has at most $c_3|P'|$ neighbors in $Q'$.  Since $G$ is $K_{k,k}$-free, there are at most $k-1$ points in $Q\setminus Q'$ that are neighbors to $p$.  Hence,

\begin{equation*}
|N_G(p)| \leq c_3|P'| + (k-1) \leq \frac{c_3}{c_2}\left(\frac{n^{\frac{2}{3}}}{m^{\frac{1}{3}}}\right) + (k-1).
\end{equation*}

\noindent  We remove $p$ and repeat this argument until there are no vertices remaining in $P$ and see that

\begin{eqnarray*}
|E(G)| & \leq &  (c/2)(n + m) + \sum\limits_{i = n^{1/2}}^m \left(\frac{c_3}{c_2}\left(\frac{n^{\frac{2}{3}}}{i^{\frac{1}{3}}}\right)  + (k-1)\right)\\\\
 & \leq & c\left(m^{\frac{2}{3}}n^{\frac{2}{3}} + m + n\right).
\end{eqnarray*}

\noindent for sufficiently large $c = c(k,t)$.

\end{proof}

\section{The case of general $d_1$ and $d_2$}\label{polyMethodSection}

The goal of this section is to prove Theorem \ref{2dim} for all dimensions $d_1,d_2$ (i.e., parts \eqref{d1d2dBd} and \eqref{d1d2GenBd} of the theorem).
\subsection{Preliminaries }\label{genDimPrelim}
We begin by introducing some tools, along with some useful notation. Our proof will use some basic tools from algebraic geometry. A nice introduction of these concepts can be found in \cite{CLOu}.

\medskip
\noindent \textbf{Real varieties.} If $V\subset \RR^d$ is a real algebraic variety, we define the dimension $\dim(V)$ of $V$ as in \cite[Section 2.8]{BCR98}. Define $V^*\subset\CC^d$ to be the \emph{complexification} of $V$---the smallest complex variety containing $V$. That is, if $\iota\colon\RR^d\to\CC^d$ is the usual embedding of $\RR^d$ to $\CC^d$, then $V^*$ is the Zariski closure (over $\CC$) of $\iota(V)$. We define $\deg(V)=\deg(V^*)$, where the latter is the usual definition of the degree of a complex variety (i.e., the cardinality of $V^*\cap H$, where $H\subset\CC^d$ is a generic flat of codimension $\dim(V^*)$).

Given a real variety $V\subset\RR^d$, we denote by $I(V)$ the ideal of polynomials $f\in\RR[x_1,\ldots,x_d]$ that vanish on $V$. We say that a real variety $V$ is \emph{irreducible} if it is irreducible over $\RR$ (see e.g.~\cite[Section 2.8]{BCR98}). In particular, if $V$ is irreducible, then $I(V)$ is a prime ideal. Moreover, for every polynomial $g\in \RR[x_1,\ldots,x_d]$ such that $g\notin I(V)$, we have that $(I(V),g)$ strictly contains $I(V)$, and thus $\dim(V\cap Z(g))<\dim V$.

\medskip

\noindent \textbf{Polynomial partitioning.} Consider a set $\pts$ of $m$ points in $\RR^d$. Given a polynomial $f \in \RR[x_1,\ldots, x_d]$, we define the \emph{zero-set} of $f$ to be
$Z(f) = \{ p \in \RR^d \mid f(p)=0 \}$.  For $1 < r \le m$, we say that
$f \in \RR[x_1,\ldots, x_d]$ is an \emph{$r$-partitioning polynomial} for
$\pts$ if no connected component of $\RR^d \setminus Z(f)$ contains more than
$m/r$ points of $\pts$.  Notice that there is no restriction on the number of
points of $\pts$ that lie in $Z(f)$.

The following result is due to Guth and Katz \cite{GK15}.
A detailed proof can also be found in \cite{KMS12}.
\begin{theorem}\label{th:partition}{\bf (Polynomial partitioning \cite{GK15})}
Let $\pts$ be a set of $m$ points in $\RR^d$. Then for every $1< r \le m$,
there exists an $r$-partitioning polynomial $f \in \RR[x_1,\ldots, x_d]$ of
degree at most $\cp \cdot r^{1/d}$, where $\cp$ depends only on $d$.
\end{theorem}

We require the following generalization of Theorem \ref{th:partition}, which we prove in Section \ref{sec:hilbert} below.

\begin{theorem}\label{th:partitionV}
Let $\pts$ be a set of $n$ points in $\RR^d$ and let $V\subset\RR^n$ be an irreducible variety of degree $D$ and dimension $d^\prime$. Then there exists an $r$-partitioning polynomial $g$ for $\pts$ such that $g\notin I(V)$ and $\deg g \le \cp \cdot r^{1/d'}$, where $\cp$ depends only on $d$ and $D$.
\end{theorem}

\subsection{Proof of Theorem \ref{2dim}}
We now establish Theorem \ref{2dim} by proving the following more general statement.
Theorem \ref{2dim} is immediately implied by Theorem \ref{moreGeneralThm}, by taking $V$ to be $\RR^{d_1}$.

\begin{theorem}\label{moreGeneralThm}
Let $G = (P,Q,E)$ be a bipartite semi-algebraic graph in $(\mathbb{R}^{d_1},\mathbb{R}^{d_2})$ such that $E$ has complexity at most $t$, $|P| = m$, and $|Q| = n$.  Moreover, let $P\subset V$, where $V\subset\RR^{d_1}$ is an irreducible variety of dimension $e$ and degree $D$. If $G$ is $K_{k,k}$-free, then for any $\eps>0$,

\begin{equation}
|E(G)| \leq  \alpha_{1,e} m^{\frac{d_2(d_1-1)}{d_1d_2-1}+\eps}n^{\frac{d_1(d_2-1)}{d_1d_2-1}}+\alpha_2(m+n),
\end{equation}
where $\alpha_{1,e},\alpha_2$ are constants that depend on $\eps,d_1,d_2,e,t,k,$ and $D$.
\end{theorem}
\begin{proof}[Proof of Theorem \ref{moreGeneralThm}]
As in Section \ref{warmup}, we think of the vertices of $P$ as points in $\RR^{d_1}$, and we think of the vertices of $Q$ as semi-algebraic sets in $\RR^{d_1}$. That is, every $q\in Q$ is the (semi-algebraic) set of all points $p\in \RR^{d_1}$ that satisfy
\[ \Phi(f_1(p,q)\geq 0,\ldots,f_t(p,q)\geq 0) = 1.\]

There is a bijection between the edges of $G$ and the incidences of $I(P,Q)$. Thus, it suffices to prove
\begin{equation} \label{eq:mainBound}
I(P,Q) \leq  \alpha_{1,e} m^{\frac{d_2(d_1-1)}{d_1d_2-1}+\eps}n^{\frac{d_1(d_2-1)}{d_1d_2-1}}+\alpha_2(m+n).
\end{equation}

We prove the theorem by a two-step induction process. First, we induct on $e$. We can be quite wasteful with each such induction step, since we perform at most $d_1$ such steps. Within every such step, we perform a second induction on $|P|+|Q| = m + n$. We must be more careful with the steps of the second induction, since we perform many such steps.

By Corollary \ref{app}, there exists a constant $C_\text{L2.3}$ (depending on $d_1,d_2,t,k$) such that $E(G)\le C_\text{L2.3}\left(mn^{1-1/d_2}+n\right)$. When $m\le n^{1/d_2}$ (and when $\alpha_2$ is sufficiently large) we have $|E(G)| \le \alpha_2 n$.
Therefore, in the remainder of the proof we assume that $n < m^{d_2}$, which implies
\begin{equation} \label{eq:nElim}
n = n^{\frac{d_1-1}{d_1d_2-1}} n^{\frac{d_1(d_2-1)}{d_1d_2-1}} \le m^{\frac{d_2(d_1-1)}{d_1d_2-1}} n^{\frac{d_1(d_2-1)}{d_1d_2-1}}.
\end{equation}
Since the conditions in Corollary \ref{app} are symmetric with respect to $d_1$ and $d_2$, we can replace $d_2$ with $d_1$ in the bound of the lemma. Thus, the same argument implies $m<n^{d_1}$ and hence
\begin{equation} \label{eq:mElim}
m \le m^{\frac{d_2(d_1-1)}{d_1d_2-1}} n^{\frac{d_1(d_2-1)}{d_1d_2-1}}.
\end{equation}

We now consider the base case for the induction.
If $m+n$ is sufficiently small, then \eqref{eq:mainBound} is immediately implied by choosing sufficiently large values for $\alpha_{1,e}$ and $\alpha_2$.
Similarly, when $e=0$, we again obtain \eqref{eq:mainBound} when $\alpha_{1,e}$ and $\alpha_2$ are sufficiently large.
\medskip

\noindent {\bf Partitioning.}
Next, we consider the induction step. That is, we assume that \eqref{eq:mainBound} holds when  $|P|+|Q| < m+n$ or $\dim V <e$.
By Theorem \ref{th:partitionV}, there exists an $r$-partitioning polynomial $f$ with respect to $V$ of degree at most $\cp \cdot r^{1/e}$, where $r$ is a large constant that will be set later on.
The dependencies between the various constants in the proof are
\[ 2^{1/\varepsilon},d_1,d_2,e,t,k,D \ll \cp, \cells, C_\text{L2.3}, \inter \ll \chol \ll r \ll \gamma_1, \alpha_{2} \ll \alpha_{1}.\]

Denote the cells of the partition as $\Omega_1, \ldots, \Omega_s$. Since we are working over the reals, there exists a polynomial $g$ whose degree depends only on $d_1$, $d_2$, and $D$ so that $Z(g)=V$. Thus, by \cite[Theorem A.2]{ST11}, there exists a constant $\cells$ such that $s\le \cells \cdot r$, where $\cells$ depends on $d_1,d_2,e,$ and $D$. We partition $I(P,Q)$ into the three following subsets:
\begin{itemize}[noitemsep]
\item $I_1$ consists of the incidences $(p,q) \in I(P,Q)$ where $p$ is contained in the variety $V \cap Z(f)$.
\item $I_2$ consists of the incidences $(p,q) \in I(P,Q)$ where $p$ is contained in a cell $\Omega$ of the partitioning, and the semi-algebraic set $q$ fully contains $\Omega$.
\item $I_3 = I(P,Q) \setminus \{I_1\cup I_2\}$. This is the set of incidences $(p,q) \in I(P,Q)$ such that $p$ is contained in a cell $\Omega$, and $q$ does not fully contain $\Omega$ (i.e., $q$ properly intersects $\Omega$).
\end{itemize}
Notice that we indeed have
\begin{equation}\label{eq:Iparts}
I(P,Q) = I_1 + I_2 + I_3.
\end{equation}

\noindent {\bf Bounding $\qvec{I}_\qvec{1}$.}
The points of $P\subset \RR^{d_1}$ that participate in incidences of $I_1$ are all contained in the variety $V' =V \cap Z(f)$. Set $m_0 = |P \cap V'|$. Since $V$ is an irreducible variety and $f \notin I(V)$, then $V'$ is a variety of dimension $e'\le e-1$. The intersection $V^\prime=V\cap Z(f)$ can be written as a union of $\gamma_1$ irreducible (over $\RR$) components, each of dimension at most $e'$ and degree at most $\gamma_2$, where $\gamma_1$ and $\gamma_2$ depend only on $D,C_{\textrm{part}},$ $d$, and $r$ (see e.g.~\cite{Gal}). We can now apply the induction hypothesis to each component to obtain
\begin{equation*}
I_1 \le \gamma_1 \alpha_{1,e-1} m_0^{\frac{d_2(d_1-2)}{(d_1-1)d_2-1}+\eps}n^{\frac{(d_1-1)(d_2-1)}{(d_1-1)d_2-1}}+\alpha_2(m_0+n).
\end{equation*}
Notice we have
\begin{equation}\label{bigComp}\begin{array}{ccl}
    m^{\frac{d_2(d_1-2)}{(d_1-1)d_2-1}}n^{\frac{(d_1-1)(d_2-1)}{(d_1-1)d_2-1}} &  = & m^{  \frac{d_2(d_1-2)}{(d_1-1)d_2-1} - \frac{d_2(d_1 - 1)}{d_1d_2 - 1} + \frac{d_2(d_1 - 1)}{d_1d_2 - 1}  }n^{   \frac{(d_1-1)(d_2-1)}{(d_1-1)d_2-1} - \frac{d_1(d_2 - 1)}{d_1d_2 - 1} +  \frac{d_1(d_2 - 1)}{d_1d_2 - 1} }  \\\\
      & =  & m^{-\frac{(d_2 - 1)d_2}{((d_1 - 1)d_2 - 1)(d_1d_2 - 1)} +\frac{d_2(d_1 - 1)}{d_1d_2 - 1}}   n^{  \frac{d_2 - 1}{((d_1 - 1)d_2 - 1)(d_1d_2 - 1)} + \frac{d_1(d_2 - 1)}{d_1d_2 - 1}} \\\\
       & = & m^{ -\frac{(d_2 - 1)d_2}{((d_1 - 1)d_2 - 1)(d_1d_2 - 1)}}n^{\frac{d_2 - 1}{((d_1 - 1)d_2 - 1)(d_1d_2 - 1)}}m^{\frac{d_2(d_1-1)}{d_1d_2-1}}n^{\frac{d_1(d_2-1)}{d_1d_2-1}}\\\\
      & \leq  & m^{\frac{d_2(d_1-1)}{d_1d_2-1}}n^{\frac{d_1(d_2-1)}{d_1d_2-1}},
  \end{array}
\end{equation}

\noindent where the last inequality follows from the fact that $m^{-d_2}n \leq 1$.  By applying \eqref{eq:nElim} to the $\alpha_2 n$ term and by choosing $\alpha_{1,e}$ to be sufficiently large with respect to $\alpha_{1,e-1}$, $\gamma_1$, and $\alpha_2$, we obtain
\begin{equation} \label{eq:I1}
I_1 \le \frac{\alpha_{1,e}}{2} m^{\frac{d_2(d_1-1)}{d_1d_2-1}+\eps}n^{\frac{d_1(d_2-1)}{d_1d_2-1}}+\alpha_2 m_0.
\end{equation}

\noindent {\bf Bounding $\qvec{I}_\qvec{2}$.}
Let $m' = m - m_0.$ This is the number of the points of $\pts$ that are not contained in $Z(f)$.
A cell of $\Omega_1,\ldots,\Omega_s$ that contains at most $k-1$ points of $P$ can yield at most $(k-1)n$ incidences.
Since $G$ is $K_{k,k}$-free, a cell that contains at least $k$ points of $P$ can be fully contained by at most $k-1$ of the semi-algebraic sets of $Q$. Since $s\le \cells \cdot r$, we obtain
\[ I_2 < \cells \cdot r \left((k-1)n + (k-1)m' \right). \]
By choosing $\alpha_2$ to be sufficiently large, we have
\begin{equation} \label{eq:I2}
I_2 \le \alpha_2 (m'+n).
\end{equation}

\noindent {\bf Bounding $\qvec{I}_\qvec{3}$.}
We say that a semi-algebraic set $q\in Q$ properly intersects a cell $\Omega$ if $q$ meets $\Omega$ but does not contain $\Omega$. For each $q\in Q$, we now bound the number of cells that $q$ properly intersects. Such a set $q$ is defined by at most $t$ equations, each of degree at most $t$. For $q$ to properly intersect a cell $\Omega$, at least one of these equations must define a variety that intersects $\Omega$ (this condition is necessary but not sufficient). Consider such an equation $E$ such that $Z(E)$ does not fully contain $V_i$ (since otherwise it would not properly intersect any cell). Since $V_i$ is irreducible, we have that the dimension of $Z(E)\cap V_i$ is at most $e-1$. Thus, by \cite[Theorem A.2]{ST11}, there exists a constant $\inter$ (depending on $t,d_1$) such that $Z(E)$ intersects at most $\inter r^{(e-1)/e}$ cells of the partition.
This in turn implies that every semi-algebraic set $q\in Q$ properly intersects at most $t\inter r^{(e-1)/e}$
cells of the partition.

For $1 \le i \le s$, we denote by $Q_i$ the set of elements of $Q$ that properly intersect the cell $\Omega_i$, and by $P_i$ the number of points of $P$ that are contained in $\Omega_i$. We set $m_i=|P_i|$ and $n_i=|Q_i|$. By the partitioning property, we have $m_i \le m/r$, for every $1 \le i \le s$. By the previous paragraph, we have
\begin{equation*}
\sum_{i=1}^s n_i \le n t\inter r^{(e-1)/e}.
\end{equation*}
By applying H\"older's inequality, we have
\begin{align*}
\sum_{i=1}^s n_i^{\frac{d_1(d_2-1)}{d_1d_2-1}} &\le \left(\sum_{i=1}^s n_i\right)^{\frac{d_1(d_2-1)}{d_1d_2-1}} \left(\sum_{i=1}^s 1\right)^{\frac{d_1-1}{d_1d_2-1}} \\ &\le \left(n t\inter r^{(e-1)/e}\right)^{\frac{d_1(d_2-1)}{d_1d_2-1}}\left(\cells \cdot r\right)^{\frac{d_1-1}{d_1d_2-1}} \\
&\le \chol n^{\frac{d_1(d_1-1)}{d_1d_2-1}}r^{1-\frac{d_1(d_2-1)}{e(d_1d_2-1)}} \\ &\le \chol n^{\frac{d_1(d_1-1)}{d_1d_2-1}}r^{1-\frac{d_2-1}{d_1d_2-1}},
\end{align*}
where $\chol$ depends on $t,\inter,\cells,d_1,d_2$.

By the induction hypothesis, we have
\begin{align}\label{inductionHypothesisComp}
\sum_{i=1}^s I(P_i,Q_i) &\le \sum_{i=1}^s \left(\alpha_{1,e} m_i^{\frac{(d_1-1)d_2}{d_1d_2-1}+\varepsilon}n_i^{\frac{d_1(d_2-1)}{d_1d_2-1}}+\alpha_{2}(m_i+n_i)\right) \\
&\le \alpha_{1,e} m^{\frac{(d_1-1)d_2}{d_1d_2-1}+\varepsilon}\Big(r^{\frac{(d_1-1)d_2}{d_1d_2-1}+\varepsilon}\Big)^{-1} \sum_{i=1}^s n_i^{\frac{d_1(d_2-1)}{d_1d_2-1}} + \sum_{i=1}^s\alpha_{2}(m_i+n_i) \\[2mm]
&= \alpha_{1,e} \chol\ r^{-\varepsilon}m^{\frac{(d_1-1)d_2}{d_1d_2-1}+\varepsilon}n^{\frac{d_1(d_2-1)}{d_1d_2-1}}   + \alpha_{2}\left(m+n t\inter r^{(e-1)/e}\right).
\end{align}
According to \eqref{eq:nElim} and \eqref{eq:mElim}, and when $\alpha_{1,e}$ is sufficiently large with respect to $r,t,\inter,\alpha_{2}$, we have
\[ \sum_{i=1}^s I(P_i,Q_i) \le 3\alpha_{1,e} \chol\ r^{-\varepsilon}m^{\frac{(d_1-1)d_1}{d_1d_2-1}+\varepsilon}n^{\frac{d_1(d_2-1)}{d_1d_2-1}}. \]
Finally, by choosing $r$ to be sufficiently large with respect to $\varepsilon,\chol$, we have
\begin{equation}\label{eq:I3}
I_3 = \sum_{i=1}^s I(P_i,Q_i) \le \frac{\alpha_{1,e}}{2} m^{\frac{(d_1-1)d_2}{d_1d_2-1}+\varepsilon}n^{\frac{d_1(d_2-1)}{d_1d_2-1}}.
\end{equation}

\noindent {\bf Summing up.} By combining \eqref{eq:Iparts}, \eqref{eq:I1}, \eqref{eq:I2}, and \eqref{eq:I3}, we obtain
\[ I(P,Q) \le \alpha_{1,e} m^{\frac{d_2(d_1-1)}{d_1d_2-1}+\eps}n^{\frac{d_1(d_2-1)}{d_1d_2-1}}+\alpha_2(m+n), \]
which completes the induction step and the proof of the theorem.
\end{proof}

\section{Hilbert polynomials and Theorem \ref{th:partitionV}}\label{sec:hilbert}

In this section, we will prove Theorem \ref{th:partitionV}. Our proof relies on Hilbert polynomials. Before presenting the proof, we begin with some algebraic preliminaries.

\subsection{Hilbert polynomials} \label{ssec:Hilbert}
Let $\RR[x_1,\ldots, x_d]_{\le m}$ be the set of polynomials of degree at most $m$ in $\RR[x_1,\ldots , x_d]$. Similarly, if $I \subset \RR[x_1,\ldots , x_d]$ is an ideal, let $I_{\le m} = I \cap \RR[x_1,\ldots , x_d]_{\le m}$ be the set of polynomials in $I$ of degree at most $m$.
It can be easily verified that there are $\binom{d+m}{m}$ monomials in $x_1,\ldots, x_d$ of degree $m$.
Thus, we can consider $\RR[x_1,\ldots, x_d]_{\le m}$ as a vector space of dimension $\binom{d+m}{m}$, and $I_{\le m}$ as a vector subspace of $\RR[x_1,\ldots, x_d]_{\le m}$. We consider a polynomial $f\in \RR[x_1,\ldots , x_d]_{\le m}$ as equivalent to any of its constant multiples $cf$ (where $c\in\RR\setminus \{0\}$), since their zero-sets are identical. Therefore, $\RR[x_1,\ldots, x_d]_{\le m}$ can be identified with the projective space $\RR\mathbf{P}^{\binom{d+m}{m}}$, and $I_{\le m}$ can be identified with a projective variety in $\RR\mathbf{P}^{\binom{d+m}{m}}$.

The quotient $\RR[x_1,\ldots , x_d]_{\le m}/I_{\le m}$ is also a vector space (see, e.g., \cite[Section 9.3]{CLOu}). The \emph{Hilbert function} of an ideal $I\subset \RR[x_1,\ldots, x_d]$ is defined as
\[h_I(m) = \dim \left(\RR[x_1,\ldots , x_d]_{\le m}/I_{\le m}\right).\]
A nice introduction to Hilbert functions can be found in \cite[Chapter 9]{CLOu}.

For every ideal $I\subset \RR[x_0,\ldots, x_d]$, there exists an integer $m_I$ and a polynomial $H_I(m)$ such that for every $m>m_I$ we have $h_I(m)=H_I(m)$. $H_I$ is called the \emph{Hilbert polynomial of} $I$, and $m_I$ is called the \emph{regularity} of $I$. We set $t= \deg H_I$, and say that the \emph{dimension} of $I$ is $t$. Notice that if $I\neq \{0\}$, then $t<d$. Let $a_I$ be the coefficient of the leading monomial of $H_I$.

If $V\subset\RR^d$ is an irreducible variety, then $\dim(I(V))=\dim(V)$, where $\dim(V)$ is defined in Section \ref{genDimPrelim}. Furthermore, the leading coefficient $a_I>0$ is bounded below by a constant $c_{d,\deg V}$ that depends only on $d$ and $\deg V$.

In \cite[Theorem B]{Giusti}, it is shown that the regularity $m_I$ of $I$ is bounded by a quantity $\tilde m$ that depends only on $d$ and $\deg V$.\footnote{It is important to note that Giusti's result in \cite{Giusti} applies in any field of characteristic 0. In particular, the field does not need to be algebraically closed. Giusti deals with homogeneous ideals, while we work with affine ideals. However, Giusti's bound also applies in the affine case. Giusti bounds the quantity $m_I$ in terms of the dimension $d$ and the maximum degree of the collection of polynomials needed to generate $I$. This quantity is in turn bounded by the degree of $V$.}

In particular, there is an integer $m^\prime$ depending only on $d$ and $\deg V$ so that for $m>m^\prime$, we have
\begin{equation}\label{lowerBdOnH}
h_{I(V)}(m)>\frac{c_{d,\deg V}}{2}m^{\dim V}.
\end{equation}
\subsection{Proof of Theorem \ref{th:partitionV}}

We first recall the discrete version of the \emph{ham sandwich theorem} (e.g., see
\cite{LMS94}).
A hyperplane $h$ in $\RR^d$ \emph{bisects} a finite point set $S\subset \RR^d$ if each of the two
open halfspaces bounded by $h$ contains at most $|S|/2$ points of $S$.
The bisecting hyperplane may contain any number of points of $S$.

\begin{theorem}{\bf (Discrete ham sandwich theorem)} \label{th:HamS}
Every $d$ finite point sets $S_1,\ldots,S_d \subset \RR^d$
can be simultaneously bisected by a hyperplane.
\end{theorem}

A polynomial $g:\mathbb{R}^d \rightarrow \mathbb{R}$ \emph{bisects} a finite
point set $S \subset \mathbb{R}^d$ if each of the two sets $\{x \in
\mathbb{R}^d: g(x)<0\}$ and $\{x \in \mathbb{R}^d: g(x)>0\}$ contains at
most $|S|/2$ points of $S$.

We combine Theorem \ref{th:HamS} with Hilbert polynomials to obtain a variant of the \emph{polynomial} ham sandwich theorem (for the original theorem, see for example \cite{GK15}).

\begin{lemma} \label{le:specialHS}
Let $V\subset\RR^d$ be an irreducible variety of dimension $d'$ and degree $D$, and let $S_1,S_2,\ldots,S_k$ be finite sets of points that are contained in $V$.
Then there exist a constant $m_0$ that depends only on $D$ and $d$, and a polynomial $g$, such that $g\notin I(V)$, $g$ bisects each of the sets $S_1,S_2,\ldots,S_k$, and
\begin{equation*}
\deg g = \left\{
\begin{array}{ll}
O_{D,d}(1), & \text{if } k < m_0, \\
O_{D,d}(k^{1/d'}),     & \text{if } k \ge m_0.
\end{array}\right.
\end{equation*}
\end{lemma}
\begin{proof}
Our proof is a variant of the proof of the polynomial ham-sandwich theorem (as presented, e.g., in \cite{GK15,KMS12}). Let $I=I(V)$.
As noted in Section \ref{ssec:Hilbert}, there exists a constant $\tilde m_I$ depending only on $d$ and $D$ so that \eqref{lowerBdOnH} holds for every $m>\tilde m_i$. Thus, the vector space $\RR[x_1,\ldots , x_d]_{\le m}/I_{\le m}$ has dimension $E_m = \Omega_{d,D}(m^{d'})$.
We choose $m$ so that $E_m \ge k$. That is,
\[ k = O_{d,D}(m^{d'}), \qquad \text{ or } \qquad m=\Omega_{d,D}(k^{1/d'}). \]
If the resulting $m$ is smaller than $\tilde m$, we replace it with $h(\tilde m) = O_{D,d}(1)$.

Let $p_1,\ldots,p_{E_m}$ be a basis for the vector space $\RR[x_1,\ldots , x_d]_{\le m}/I_{\le m}$. For each $i=1,\ldots,E_m$, choose a representative $\tilde p_i\in \RR[x_1,\ldots , x_d]_{\le m}$ which lies in the equivalence class $p_i$. We will choose $\tilde p_i$ to be of smallest possible degree (note that the choice of $\tilde p_i$ need not be unique). Consider the polynomial mapping $\phi: \RR^d \to \RR\mathbf{P}^{E_m}$ defined by
 \[ \phi(x) = (\tilde p_1(x),\ldots,\tilde p_{E_m}(x)).\]

For every $1\le i \le k$, let $S_i' = \phi(S_i) \subset \RR\mathbf{P}^{E_m}$. Note that $\phi$ is injective on $V=Z(I)$, and thus $|S_i'|=S_i$. By Theorem \ref{th:HamS}, there exists a hyperplane $h \subset \RR\mathbf{P}^{E_m}$ that bisects each of the sets $S'_1,S'_2,\ldots,S'_k$. The hyperplane $h$ can be defined as $Z(a_1y_1+\ldots+a_{E_m}y_{E_m})$ for some $a_1,\ldots,a_{E_m}\in \RR$. In other words, for each $i=1,\ldots,k$, we have
\begin{equation*}
|\{y\in S^\prime_i \colon a_1y_1+\ldots+a_{E_m}y_{E_m}>0\}|\leq |S^\prime_i|/2,
\end{equation*}
and
\begin{equation*}
|\{y\in S^\prime_i \colon a_1y_1+\ldots+a_{E_m}y_{E_m}<0\}|\leq |S^\prime_i|/2.
\end{equation*}
If $x\in \RR^d$, then $a_1\tilde p_1(x)+\ldots + a_{E_m}\tilde p_{E_m}(x)=(a_1\tilde p_1+\ldots+a_{E_m}\tilde p_{E_m})(x).$ Thus, if we let $g=a_1\tilde p_1+\ldots+a_{E_m}\tilde p_{E_m}$, then $g$ is a polynomial of degree at most $m$, $g\notin I,$ and for each $i=1,\ldots,E_m$,
\begin{equation*}
\begin{split}
&|\{y\in S \colon g(y)>0\}|\leq |S|/2,\\
&|\{y\in S \colon g(y)<0\}|\leq |S|/2,
\end{split}
\end{equation*}
i.e.,~$g$ bisects each of the sets $S_1,\ldots,S_{E_m}.$
\end{proof}

The standard polynomial partitioning theorem is proved by using the polynomial ham-sandwich theorem.
Our variant of the polynomial partitioning theorem is proved by using our variant of the polynomial ham-sandwich theorem (i.e., Lemma \ref{le:specialHS}). We now recall the statement of Theorem \ref{th:partitionV}, and then prove it.
\begin{partitionVThmEnv}
Let $P$ be a set of $n$ points in $\RR^d$ and let $V\subset \RR^d$ be an irreducible variety of dimension $d'$ and degree $D$. Then there exists an $r$-partitioning polynomial $g$ for $P$ such that $g\notin I(V)$ and $\deg g = O(r^{1/d'})$. The implicit constant depends only on $D$ and $d$.
\end{partitionVThmEnv}

\begin{proof}
In this section, all logarithms will be base 2. Let $m_0$ be the constant specified in Lemma \ref{le:specialHS}. Let $c_D$ denote the constant in the bound of Lemma \ref{le:specialHS} for the case $k < m_0$ and let $c_1$ be the constant hidden in the $\Omega$-notation of the case $k \ge m_0$. Finally, let $c_2=c_1/(1-1/2^{1/d'})$.

Let $I=I(V)$. We show that there exists a sequence of polynomials $g_0,g_1,g_2,\ldots$ with the following properties
\begin{itemize}[noitemsep]
\item $g_i \notin I$
\item For $0 \le i < \log m_0,$  $\deg g_i\leq i \cdot c_D$. For $i\ge \log m_0$, $\deg g_i\leq c_D\log m_0  + c_2 2^{i/d'}$.
\item Every connected component of $\RR^d\setminus Z(g_i)$ contains at most $m/2^i$ points of $\pts$.
\end{itemize}
If we can find such a sequence of polynomials, we can complete the proof of the theorem by setting $t=\lceil \log r \rceil$ and taking $g=g_t$.

We prove the existence of $g_0,g_1,g_2,\ldots$ by induction.
For the base case of the induction, let $g_0=1$.
For the case $1 \le i < \log m_0$, by the induction hypothesis there exists a polynomial $g_{i-1}$ of degree at most $(i-1)c_D$ such that every connected component of $\RR^d\setminus Z(g_{i-1})$ contains at most $m/2^{i-1}$ points of $\pts$.
Since $|\pts|=m$, the number of these connected components that contains more than $m/2^{i}$ points of $\pts$ is smaller than $2^{i}$. Let $S_1,\ldots,S_n \subset \pts$ be the subsets of $\pts$ that are contained in each of these connected components (that is, $|S_i|>m/2^{i}$ for each $i$, and $n<2^{i}$). By Lemma \ref{le:specialHS}, there is a polynomial $h_{i-1}\notin I$ of degree smaller than $c_0$ that simultaneously bisects every $S_i$. We can set $g_{i} = g_{i-1} \cdot h_{i-1}$, since every connected component of $\RR^d\setminus Z(g_{i-1} \cdot h_{i-1})$ contains at most $m/2^{i}$ points of $\pts$ and $g_{i-1} \cdot h_{i-1}$ is a polynomial of degree smaller than $ic_D$. Moreover, since $I$ is a prime ideal that does not contain $g_{i-1}$ and $h_{i-1}$, it also does not contain $g_{i-1} \cdot h_{i-1}$.

Next, we consider the case $\log m_0 \le i$, and analyze it similarly. That is, by the induction hypothesis there exists a polynomial $g_{i-1} \notin I$ of degree smaller than $\log m_0 c_D + c_2 2^{(i-1)/d'}$ such that every connected component of $\RR^d\setminus Z(g_{i-1})$ contains at most $m/2^{i-1}$ points of $\pts$.
Since $|\pts|=m$, the number of these connected components that contain more than $m/2^{i}$ points of $\pts$ is smaller than $2^{i}$. Let $S_1,\ldots,S_n \subset \pts$ be the subsets of $\pts$ that are contained in each of these connected components (that is, $|S_i|>m/2^{i}$ for each $i$, and $n<2^{i+1}$). By Lemma \ref{le:specialHS}, there is a polynomial $h_{i-1}\notin I$ of degree smaller than $c_1 2^{i/d'}$ that simultaneously bisects every $S_i$. We can set $g_{i} = g_{i-1} \cdot h_{i-1}$, since every connected component of $\RR^d\setminus Z(g_{i-1} \cdot h_{i-1})$ contains at most $m/2^{i}$ points of $\pts$.
Moreover, $g_{i-1} \cdot h_{i-1}$ is a polynomial of degree smaller than
\begin{align*}
c_D\log m_0  + c_2 2^{(i-1)/d'} + c_1 2^{i/d'} &= c_D\log m_0 + 2^{i/d'} \left(\frac{c_2}{2^{1/d'}} + c_1 \right) \\
&= c_D\log m_0  + c_2 2^{i/d'}.
\end{align*}
This completes the induction step, and thus also the proof of the theorem.
\end{proof}

\section{Applications}\label{applicationSection}

\subsection{Incidences with algebraic varieties in $\RR^d$}
The following theorem is a variant of a well known incidence bound in the plane.
\begin{theorem} \label{th:PS} {\bf (Pach and Sharir \cite{PS92,PS98})}
Let $\pts$ be a set of $m$ points and let $\curves$ be a set of $n$ constant-degree algebraic curves, both in $\RR^2$, such that the incidence graph of $\pts\times\curves$ does not contain a copy of $K_{s,t}$.
Then
$$I(\pts,\curves) = O\left(m^{s/(2s-1)}n^{(2s-2)/(2s-1)}+m+n\right),$$
where the implicit constant depends on $s,t,$ and the maximum degree of the curves.
\end{theorem}
While this bound is not tight for many cases, such as incidences with circles or with parabolas, it is the best known \emph{general} incidence bound in $\RR^2$. Building off the results in \cite{clarkson}, Zahl introduced the following three-dimensional variant of this bound:.
\begin{theorem} \label{th:ZahlR3} {\bf (Zahl \cite{Zahl13})}
Let $\pts$ be a set of $m$ points and let $\vrts$ be a set of $n$ smooth constant-degree algebraic varieties, both in $\RR^3$, such that the incidence graph of $\pts\times\vrts$ does not contain a copy of $K_{s,t}$. Then
$$I(\pts,\vrts) = O\left(m^{2s/(3s-1)}n^{(3s-3)/(3s-1)}+m+n\right),$$
where the implicit constant depends on $s,t,$ and the maximum degree of the varieties.
\end{theorem}

Very recently, Basu and Sombra \cite{BasuSombra14} obtained a similar bound in $\RR^4$.
\begin{theorem} \label{th:Basombra} {\bf (Basu and Sombra \cite{BasuSombra14})}
Let $\pts$ be a set of points and let $\vrts$ be a set of constant-degree algebraic varieties, both in $\RR^4$, such that the incidence graph of $\pts\times\vrts$ does not contain a copy of $K_{s,t}$. Then
$$I(\pts,\vrts) = O\left(|\pts|^{3s/(4s-1)}|\vrts|^{(4s-4)/(4s-1)}+|\pts|+|\vrts|\right),$$
where the implicit constant depends on $s,t,$ and the maximum degree of the varieties.
\end{theorem}

When looking at the bounds of Theorems \ref{th:PS}, \ref{th:ZahlR3}, and \ref{th:Basombra}, a pattern emerges. An easy variant of our technique in Sections \ref{polyMethodSection} and \ref{sec:hilbert} yields the bound coming from this pattern (up to an extra $\eps$ in the exponent).

\begin{IncRdThmEnv}
Let $\pts$ be a set of $m$ points and let $\vrts$ be a set of $n$ constant-degree algebraic varieties, both in $\RR^d$, such that the incidence graph of $\pts\times\vrts$ does not contain a copy of $K_{s,t}$ (here we think of $s,t,$ and $d$ as being fixed constants, and $m$ and $n$ are large). Then for every $\eps>0$, we have
$$I(\pts,\vrts) = O\left(m^{\frac{(d-1)s}{ds-1}+\eps}n^{\frac{d(s-1)}{ds-1}}+m+n\right).$$
\end{IncRdThmEnv}

Theorem \ref{th:IncRd} improves upon a weaker bound that was obtained by Elekes and Szab\'o \cite{ESz12}. In addition to generalizing Theorem \ref{th:PS}, Theorem \ref{th:ZahlR3}, and Theorem \ref{th:Basombra}, Theorem \ref{th:IncRd} generalizes various other incidence bounds to $\RR^d$ (again, up to an extra $\eps$ in the exponent). For example, Edelsbrunner, Guibas, and Sharir \cite{EGS90} considered point-plane incidences in $\RR^3$, where no three points are collinear. Theorem \ref{th:IncRd} generalizes this result to $\RR^d$, where no $d$ points are contained in a common $(d-2)$-flat. A further generalization is to other types of hypersurfaces, such as spheres.

As shown in \cite{Sheffer15}, when $s=2$ Theorem \ref{th:IncRd} is tight up to subpolynomial factors. Specifically, \cite{Sheffer15} presents lower bounds for the cases of hyperplanes, hyperspheres, and paraboloids with no $K_{2,t}$ in the incidence graph.

\begin{proof}[Proof sketch of Theorem \ref{th:IncRd}]
The proof is very similar to the proof of Theorem \ref{moreGeneralThm}, so here we only explain how to change the original proof. As with Theorem \ref{moreGeneralThm}, Theorem \ref{th:IncRd} comes from the following generalization.
\begin{theorem}\label{varietiesInductThm}
Let $\pts$ be a set of points and let $\vrts$ be a set of constant-degree algebraic varieties, both in $\RR^d$, with $|\pts|=m,\ |\vrts|=n$, such that the incidence graph of $\pts\times\vrts$ does not contain a copy of $K_{s,t}$. Suppose that $\pts$ is fully contained in an irreducible variety $V$ of dimension $e$ and degree $D$. Suppose furthermore that no surface $S\in \vrts$ contains $V$. Then for every $\eps>0$, we have
\begin{equation}\label{IBdVarieties}
I(\pts,\vrts) \leq \alpha_{1,e}m^{\frac{(e-1)s}{es-1}+\eps}n^{\frac{e(s-1)}{es-1}}+\alpha_{2,e}(m+n),
\end{equation}
where $\alpha_{1,e}$ and $\alpha_{2,e}$ are constants that depend on $\eps, d, e, s, t,$ and $D$.
\end{theorem}

The proof of Theorem \ref{varietiesInductThm} parallels that of Theorem \ref{moreGeneralThm}. As in the proof of Theorem \ref{moreGeneralThm}, we will induct both on $e$ (the dimension of the variety $V$) and on the quantity $m+n$. As before, we can be very wasteful when we induct on $e$, but we must be more efficient when we induct on $m+n$.

As before, we find an $r$-partitioning polynomial $f$. The main difference in the proofs is that we replace Corollary \ref{app} with the K\H ov\' ari-S\'os-Tur\'an theorem (see, e.g. \cite[Section 4.5]{Mat02}). This allows us to still have \eqref{eq:nElim}, but not \eqref{eq:mElim}. To overcome this difficulty, we first change the way that the incidences are partitioned into three subsets $I_1,I_2,I_3$:
\begin{itemize}[noitemsep]
\item $I_1$ consists of the incidences $(p,S) \in \pts\times\vrts$ such that $p\in Z(f)$ and $S$ properly intersects every irreducible component of $V \cap Z(f)$ that contains $p$.
\item $I_2$ consists of the incidences $(p,S) \in \pts\times\vrts$ such that $p$ is contained in an irreducible component of $V \cap Z(f)$ that is fully contained by $S$.
\item $I_3 = I(P,Q) \setminus \{I_1\cup I_2\}$. This is the set of incidences $(p,S) \in \pts\times\vrts$ such that $p$ is not contained in $V \cap Z(f)$.
\end{itemize}

Let $V^\prime = V \cap Z(f)$ and let $m_0 = |\pts \cap V^\prime|.$ To bound $I_1$, we argue as in the proof of Theorem \ref{moreGeneralThm}. This is an incidence problem on the variety $V^\prime = V\cap Z(f)$, which has dimension at most $e-1$. Arguing as in Theorem \ref{moreGeneralThm}, we obtain the bound
$$
I_1 \leq  C\alpha_{1,e-1}m_0^{\frac{(e-2)s}{(e-1)s-1}+\eps}n^{\frac{(e-1)(s-1)}{(e-1)s-1}}+\alpha_{2,e-1}(m_0+n),
$$
where the constant $C$ depends on $D,d,e,s,t$ and the degree of $f$. A computation analogous to \eqref{bigComp} shows that this is at most
\begin{equation}\label{I1BdVarieties}
\frac{\alpha_{1,e}}{3}m^{\frac{(e-1)s}{es-1}+\eps}n^{\frac{e(s-1)}{es-1}}+\frac{\alpha_{2,e}}{2}m_0,
\end{equation}
provided $\alpha_{1,e}$ and $\alpha_{2,e}$ are chosen sufficiently large.

Let $m^\prime = m-m_0$. The proof bounding $I_3$ proceeds exactly as the proof in Theorem \ref{moreGeneralThm}, and we obtain the bound
\begin{equation}\label{I3BdVarieties}
I_3 \leq \frac{\alpha_{1,e}}{3}m^{\frac{(e-1)s}{es-1}+\eps}n^{\frac{e(s-1)}{es-1}}+\alpha_{2,e}m^\prime,
\end{equation}
which is the analogue of \eqref{inductionHypothesisComp}. Note that \eqref{I3BdVarieties} should also include the term $\alpha_{2,e}n$, but \eqref{eq:nElim} allows us to combine this with the first term in \eqref{I3BdVarieties}.

It remains to bound $I_2$. As in the discussion in Theorem \ref{moreGeneralThm} for bounding the quantity $I_1$, note that $V^\prime$ can be written as a union of $\gamma_1$ irreducible (over $\RR$) components, each of dimension at most $e-1$, where $\gamma_1$ depends only on $D$, $e$, $d$ and the degree of $f$. Since the incidence graph of $\pts\times\vrts$ does not contain a copy of $K_{s,t}$, each of the irreducible components of $V^\prime$ either contains at most $s$ points, or it is contained in at most $t$ surfaces from $\vrts$. The contribution from the first quantity is at most $s \gamma_1 n$, and the contribution from the second quantity is at most $tm_0$. Combining these bounds with \eqref{eq:nElim}, we have

\begin{equation}\label{I2BdVarieties}
I_2 \le \frac{\alpha_{1,e}}{3}m^{\frac{(e-1)s}{es-1}+\eps}n^{\frac{e(s-1)}{es-1}}+\frac{\alpha_{2,e}}{2}m_0,
\end{equation}
provided we select $\alpha_{2,e}\geq 2 t$. Combining \eqref{I1BdVarieties}, \eqref{I3BdVarieties}, and \eqref{I2BdVarieties} gives us \eqref{IBdVarieties}, which establishes Theorem \ref{varietiesInductThm} and in turn Theorem \ref{th:IncRd}.
\end{proof}

\subsection{Unit distances in $\RR^d$}

For a finite set $\pts \subset \RR^d$, we denote \emph{the number of unit distances that are spanned by} $\pts$ as the number of pairs $p,q\in \pts^2$ such that $|p-q|=1$ (where $|p-q|$ denotes the Euclidean distance between $p$ and $q$).
Let $f_d(n)$ denote the maximum number of unit distances that can be spanned by a set of $n$ points in $\RR^d$. The unit distances problem, first posed by Erd\H{o}s in \cite{Erdos1,Erdos2}, asks for the asymptotic behavior of $f_{2}(n)$ and $f_{3}(n)$. Currently, the best known bounds are $f_2(n) = O(n^{4/3})$ \cite{Spencer}, $f_2(n) = n^{1+\Omega(1 / \log\log n)}$ \cite{Erdos1}, $f_3(n) = O(n^{3/2})$ \cite{kaplan, Zahl13}, and $f_3(n)=\Omega(n^{4/3}\log\log n)$ \cite{Erdos2}. For any $d\ge 4$, we have the trivial bound $f_d(n) = \Theta(n^2)$ (see, e.g., \cite{Lenz55}). For example, in $\RR^4$, let $P_1$ be a set of $n/2$ points arranged on the circle $x_1^2+x_2^2=1/2$, and let $P_2$ be a set of $n/2$ points arranged on the circle $x_3^3+x_4^2 = 1/2$. Then, if $P=P_1\cup P_2$, the set $P$ has at least $n^2/4$ unit distances.

The problem in $d\ge 4$ becomes non-trivial once we consider only point sets with some restriction on them. Oberlin and Oberlin \cite{Oberlin12} obtain an improved upper bound under a natural restriction, as follows.

\begin{theorem} \label{th:oberlin}{\bf (\cite{Oberlin12})}
Let $d\ge 2$ and consider an $n$-point set $\pts\subset \RR^d$, such that no $d$-element subset of $\pts$ is contained in a $(d-2)$-flat. Then the number of unit distances that are spanned by $\pts$ is $O(|P|^{(2d-1)/d})$.
\end{theorem}

We now improve Theorem \ref{th:oberlin} by applying Theorem \ref{2dim}. First, we show that the configuration described above is essentially the only one that yields $\Theta(n^2)$ unit distances in $\RR^4$. Call two circles $(C_1,C_2)$ ``a pair of orthogonal circles of radius $1/\sqrt{2}$'' if (after a translation and rotation) they are the two circles $x_1^2+x_2^2=1/2,\ x_3^2+x_4^2=1/2$.

\begin{theorem}[\bf Unit distances in $\RR^4$] \label{th:UnitR4}
Let $\pts$ be a set of $n$ points in $\RR^4$, so that for any pair of orthogonal circles of radius $1/\sqrt{2}$, one of the circles contains fewer than $k$ points (for some constant $k$). Then, for any $\eps>0$, the number of unit distances spanned by $\pts$ is $O(n^{8/5+\eps})$.
\end{theorem}
\begin{proof}
Consider the bipartite graph whose vertex set consists of two copies of $\pts$, and where an edge $(p,q)$ exists if and only if $|p-q|=1$. This is a semi-algebraic bipartite graph in $(\RR^4,\RR^4)$. If we can also show that this graph contains no copy of $K_{k,k}$, then by Theorem \ref{2dim} the number of edges (i.e., the number of unit distances) is as stated in the theorem.

Without loss of generality, we may assume that $k>2$. We assume, for contradiction, that there exist two collections of points $p_1,\ldots,p_{k} \subset \pts$ and $q_1,\ldots,q_{k}\subset \pts$ such that $|p_i-q_j|=1$ for all indices $1\le i,j \le k$. That is, if $S_i$ denotes the unit sphere centered at $p_i$, then $q_1,\ldots,q_{k} \in \bigcap_{i=1}^{k} S_{i},$ which implies $|\bigcap_{i=1}^{k} S_{i}|>2$. The intersection of at least three unit hyperspheres cannot be a 2-dimensional sphere. Moreover, if such an intersection is zero-dimensional, then it consists of at most two points. Therefore, $\bigcap_{i=1}^{k} S_{i}$ must be a circle. Similarly, if $S'_i$ denotes the unit sphere centered at $q_i$, then $\bigcap_{i=1}^{k} S'_{i}$ must be a circle that contains $p_1,\ldots,p_k$. Elementary geometry then shows that these two circles must be a pair of orthogonal circles of radius $1/\sqrt{2}$, which contradicts the assumption concerning orthogonal circles of radius $1/\sqrt{2}$, and thus completes the proof.
\end{proof}
Notice that Theorem \ref{th:UnitR4} implies a better bound than Theorem \ref{th:oberlin}, while also relying on a weaker assumption. We now present a general bound for any $d$, where we have a similar assumption to the one in Theorem \ref{th:UnitR4}, though with an improved bound.
\begin{theorem}[\bf Unit distances in $\RR^d$] \label{th:UnitRd}
Let $\pts$ be a set of $n$ points in $\RR^d$, so that every $(d-3)$-dimensional sphere contains fewer than $k$ points (for some constant $k$). Then, for any $\eps>0$, the number of unit distances spanned by $\pts$ is $O(n^{2d/(d+1)+\eps})$.
\end{theorem}
\begin{proof}
As before, we consider the semi-algebraic bipartite graph whose vertex set consists of two copies of $\pts$, and where an edge $(p,q)$ exists if and only if $|p-q|=1$. This time, this graph is in $(\RR^d,\RR^d)$. If we can show that this graph contains no copy of $K_{k,k}$, then Theorem \ref{2dim} would imply that the number of edges (i.e., the number of unit distances) is as stated in the theorem.

Without loss of generality, we may assume that $k>2$. We assume, for contradiction, that there exist two collections of points $p_1,\ldots,p_{3} \subset \pts$ and $q_1,\ldots,q_{k}\subset \pts$ such that $|p_i-q_j|=1$ for all indices $1 \le i \le 3$ and $1\le j \le k$. That is, if $S_i$ denotes the unit hypersphere centered at $p_i$, then $q_1,\ldots,q_{k} \in \bigcap_{i=3}^{k} S_{i}$.
The intersection $S_1 \cap S_2$ is fully contained in the perpendicular bisector $\pi_{12}$ of $p_1$ and $p_2$, and the intersection $S_1 \cap S_3$ is fully contained in the perpendicular bisector $\pi_{13}$ of $p_1$ and $p_3$. Since $\pi_{12} \neq \pi_{13}$, the intersection of these two hyperplanes is a $(d-2)$-dimensional flat, and thus $q_1,\ldots,q_k \subset \pi_{12} \cap \pi_{13} \cap S_1$. This intersection is a $(d-3)$-dimensional sphere, contradicting the assumption of the theorem.
\end{proof}

\subsection{Incidences between points and tubes}
As an immediate corollary to Theorem \ref{2dim}, we establish the following bound on the number of incidences between points and tubes in $\RR^d$ (where a tube is the set of all points of distance at most $\delta$ from a given line). Notice that the set of tubes in $\RR^d$ can be parameterized using $2d-2$ parameters.

\begin{corollary}
For $\delta > 0$, let $\pts$ and $\Sigma$ be a set of $m$ points and $n$ tubes in $\RR^d$ respectively, such that each tube has a radius of $\delta$.  If the incidence graph contains no copy of $K_{k,k}$, then for any $\eps>0$ we have
\[ I(\pts,\Sigma) = O\left(m^{\frac{(2d-2)(d-1)}{d(2d-2)-1}+\eps}n^{\frac{d(2d-3)}{d(2d-2)-1}} + m + n\right). \]
\end{corollary}

\noindent In the planar case, we get $O(m^{2/3}n^{2/3}+m+n)$ incidences between points and strips.

\subsection{Incidences with $k$-dimensional families of varieties}
For each integer $D\geq 0$, let $\RR[x_1,\ldots,x_d]_{\leq D}$ be the vector space polynomials of degree at most $D$. As in Section \ref{ssec:Hilbert}, $\RR[x_1,\ldots,x_d]_{\leq D}$ can be identified with $\RR\mathbf{P}^{\binom{d+D}{D}}$ (here, as in Section \ref{ssec:Hilbert}, we identify a polynomial $f$ with all polynomials of the form $cf$ with $c\in\RR\backslash\{0\}$). If $\mathcal{M}\subset\RR\mathbf{P}^{\binom{d+D}{D}}$, we say that the polynomial $f$ is an element of $\mathcal M$ if the equivalence class of $f$ (where, as above $f$ is identified with every polynomial of the form $cf, c\in\RR\backslash\{0\}$) is an element of $\mathcal{M}$.

Recently, Wang, Yang, and Zhang \cite{WYZ13} derived the following result (our formulation of the theorem is somewhat different from the one in \cite{WYZ13}).
\begin{theorem} \label{th:WYZ}
Let $\mathcal{M}\subset\RP^{\binom{D+2}{2}}$ be an algebraic variety of dimension $k$. Let $\pts$ be a set of $m$ points in the plane, and let $\curves\subset\mathcal{M}$ be a set of $n$ plane curves of degree at most $D$, each of which is the zero-set of a polynomial that lies in $\mathcal{M}$. Suppose that no two curves of $\curves$ share a common component.
Then
\[ I(\pts,\curves) = O\left(m^{k/(2k-1)}n^{(2k-2)/(2k-1)}+m+n\right). \]
The implicit constant depends on $k$ and $\mathcal{M}$.
\end{theorem}

For example, assume that $\curves$ is a set of curves of degree at most $D$. By
B\'ezout's theorem (e.g., see \cite{CLOu}), two such curves can intersect in at most $D^2$ points, so we can apply Theorem \ref{th:PS} with $k=D^2+1$. However, since a bivariate polynomial of degree $D$ has at most $\binom{D+2}{2}=(D+2)(D+1)/2$ monomials, we obtain an improved bound by applying Theorem \ref{th:WYZ} with $\mathcal{M}=\RP^{\binom{D+2}{2}}$ and $k=(D+2)(D+1)/2$.

Using the techniques developed in this paper, we can extend this result to higher dimensions (though we have an $\eps$ loss in the exponents).
\begin{theorem}\label{corOf2dim}
Fix integers $d$ and $D$. Let $\mathcal{M}\subset\RP^{\binom{D+d}{d}}$ be an algebraic variety of dimension $k$. Let $\pts$ be a set of $m$ points in $\RR^d$, and let $\vrts$ be a set of $n$ algebraic varieties in $\RR^k$, each of which is the zero-set of some polynomial of degree at most $D$ that lies in $\mathcal{M}$. Suppose that the incidence graph contains no copy of $K_{s,s}$ for some constant $s$.
Then
\begin{equation}\label{corDesiredBd}
I(\pts,\vrts) = O\left(m^{\frac{k(d-1)}{(dk-1)}+\eps}n^{\frac{d(k-1)}{(dk-1)}}+m+n\right),
\end{equation}
where the implicit constant depends on $\mathcal{M},D,d,$ and $s$.
\end{theorem}
\begin{proof}
We need to rephrase this problem in a form that can be addressed by Theorem \ref{moreGeneralThm}. The idea is that we will find a variety $V\subset\RR^{\binom{D+d}{d}}$ of dimension $k$, a collection $Q\subset V$ of $n$ points lying on $V$, and a collection $\mathcal Z$ of $m$ bounded-degree algebraic in varieties in $\RR^{\binom{D+d}{d}}$ so that the incidence graph of $Q\times\mathcal Z$ has no $K_{s,s}$, and $I(Q,\mathcal Z) = I(\pts,\vrts)$.

To each variety $S\in\vrts$, we can associate an affine polynomial $f_S\in \mathcal{M}$. Without loss of generality, we can assume that none of these polynomials lies on the hyperplane $\{x_0=0\}$ (we can always apply a linear change of coordinates to guarantee that this is the case). Let $\tilde S\subset\RR[x_1,\ldots,x_d]$ be the dehomogenization of $S$ with respect to the coordinate chart $\{x_0\neq 0\}\subset\RR^{\binom{D+d}{d}}.$

For each $p\in\RR^d$, let $V_p$ be the dehomogenization of the (homogeneous) variety $\{f\in\RR[x_1,\ldots,x_d]_{\leq D}=\RR^{\binom{D+d}{d}}\colon f(p)=0\}$ with respect to the coordinate chart $\{x_0\neq 0\}\subset\RR^{\binom{D+d}{d}}.$ Then $f_S\in V_p$ if and only if $p\in S$. Furthermore, if $Q=\{f_S\colon S\in\vrts\}$ and $\mathcal Z = \{V_p\colon p\in\pts\}$, then $Q\subset\tilde{\mathcal{M}}$, and the incidence graph of $Q\times\mathcal Z$ is the same as the incidence graph of $\vrts\times\pts$. In particular, this implies that the former incidence graph contains no $K_{s,s}$, and $I(Q,\mathcal Z) = I(\pts,\vrts).$

We may now apply Theorem \ref{moreGeneralThm} to conclude that
\begin{equation}\label{incidenceBdWrongEps}
I(\pts,\vrts) =  O\left(m^{\frac{k(d-1)}{(dk-1)}}n^{\frac{d(k-1)}{(dk-1)}+\frac{\eps}{k}}+m+n\right).
\end{equation}
Now, if $n> m^k$ then Theorem \ref{corOf2dim} follows immediately from Corollary \ref{app}. If $n\leq m^k$, then \eqref{incidenceBdWrongEps} implies \eqref{corDesiredBd}. In either case, Theorem \ref{corOf2dim} is proved.
\end{proof}

\subsection{A variant for semi-algebraic hypergraphs}

As our final application, we briefly mention that one can also obtain a version of Theorem \ref{2dim} for $r$-regular $r$-partite semi-algebraic hypergraphs. That is, let $G = (P_1,\ldots,P_r,E)$ be a semi-algebraic $r$-partite graph in $(\mathbb{R}^{d_1},\RR^{d_2},\ldots,\RR^{d_r})$, and set $d=\sum_i d_i$. Then there are polynomials $f_1,f_2,\ldots,f_{t} \in \mathbb{R}[x_1,\ldots,x_{d}]$ and a Boolean function $\Phi(X_1,X_2,\ldots,X_{d})$ such that for $(p_1,\ldots,p_r) \in P_1\times \cdots \times P_r \subset \mathbb{R}^{d}$,

\begin{equation*}
(p_1,\ldots,p_r) \in E \hspace{.5cm}\Leftrightarrow \hspace{.5cm}\Phi\big(f_1(p_1,\ldots,p_r) \geq 0,\ldots,f_t(p_1,\ldots,p_r) \geq 0\big) = 1.
\end{equation*}

To bound the number of edges of such a graph, we partition the set $\{ 1,\ldots,r \}$ into two disjoint subsets $S_1,S_2$, and then set $P=\prod_{i\in S_1} P_i$ and $Q=\prod_{j\in S_2} P_j$ (i.e., Cartesian products of the point sets). We can then apply Theorem \ref{2dim} on the bipartite semi-algebraic graph of $P$ and $Q$.

In Theorem \ref{2dim}, we needed to assume that the graph was $K_{k,k}$ free. Here, the analogous requirement is that the graph $G$ be $K_{k,\ldots,k}$-free for some constant $k$.

\begin{corollary}
Let $G = (P_1,\ldots,P_r,E)$ be an $r$-regular $r$-partite semi-algebraic hypergraph in $(\mathbb{R}^{d_1},\RR^{d_2},\ldots,\RR^{d_r})$, such that $E$ has description complexity at most $t$.

For any subset $S\subset \{1,2,\ldots,r \}$, we set $m=m(S) = \prod_{i\in S} |\pts_i|$, $n=n(S) = \prod_{i\notin S} |\pts_i|$, $d_{1}=d_1(S) = |S|$, and $d_{2}=d_2(S) = r-|S|$.
If $G$ is $K_{k,\ldots,k}$-free, then
\begin{align}
&|E(G)| \leq \min_{S\subset \{1,\ldots,r\} \atop |S| \neq 0,r } \left\{c_3 \left(m^{\frac{d_{2}(d_{1}-1)}{d_{1}d_{2}-1}+\eps}n^{\frac{d_{1}(d_{2}-1)}{d_{1}d_{2}-1}} + m + n\right)\right\}.\label{d1d2GenBd}
\end{align}

\noindent Note that in the expression above, the quantities $m,n,d_1,$ and $d_2$ depend on $S$, but we have suppressed that dependence to improve readability. Here, $\eps$ is an arbitrarily small constant and $c_1 = c_1(t,k), c_2 = c_2(d_1,\ldots,d_r,t,k,\eps)$, and $c_3 = c_3(d_1,\ldots,d_r,t,k,\eps)$.
\end{corollary}

\section{Discussion} \label{sec:disc}

The main open question that arises from this work seems to be whether Theorem \ref{2dim} is tight. The only lower bounds that we are aware of arise from incidence problems with algebraic objects (mainly point-hyperplane incidence problems). These lead to a tight bound only for the case where $d_1=d_2=2$. In any other case, it would be rather interesting to find configurations of semi-algebraic objects that yield an asymptotically larger number of incidences. It is also possible that more sophisticated configurations of algebraic objects (possibly even hyperplanes) suffice to obtain tight bounds for other values of $d_1$ and $d_2$.
However, for some incidence problems (such as point-circle incidences in $\RR^2$, which corresponds to $d_1=2$ and $d_2=3$) better bounds are known than the ones implied by Theorem \ref{2dim}. This might hint that Theorem \ref{2dim} is not tight for various values of $d_1$ and $d_2$.

On the other hand, Theorem \ref{weak} is tight for $m=n$, as implied by the constructions from Koll\'ar, R\'onyai, and Szab\'o \cite{KRS96} and Alon, R\'onyai, Szab\'o \cite{ARS99}. These works prove, for each fixed $d$, the existence of bipartite graphs with both parts of size $n$,
$\Omega(n^{2-1/d})$ edges, and no copy of $K_{d,t}$ for $t=(d-1)!+1$. It is not hard to verify that such graphs satisfy $\pi_{\mathcal{F}}(z) \leq cz^d$ for all $z$. Indeed, for each $d$-set, there are at most $t-1$ vertices whose neighborhood is a superset of that $d$-set, and each
vertex with $D$ neighbors in a set of size $z$ gives rise to $\binom{D}{d}$ subsets of
size $d$. This implies that if $G$ is $K_{d,t}$-free, then $\pi_{\mathcal{F}}(z)$ is at most
$(t-1)\binom{z}{d}+\binom{z}{d-1}+\ldots+\binom{z}{0} \leq cz^d$ for an
appropriate choice of $c=c(d,t)$. A more precise bound, showing that the worst
case is given by considering the projections being sets of size at most $d$,
with an additional $(t-2)\binom{z}{d}/(d+1)$ projections of size $d+1$, gives
an upper bound of $\pi_{\mathcal{F}}(z) \leq \frac{d+t-1}{d+1}\binom{z}{d}+\binom{z}{d-1}+\binom{z}{d-2}+\ldots+\binom{z}{0}$.

\end{document}